\documentclass[12pt,reqno]{amsart}
\usepackage{amsmath,amssymb,amsfonts,amsthm,amscd,amstext,amsxtra,amsopn,array,url,mathrsfs,enumerate,anysize,enumitem,mathabx}
\marginsize{2.20cm}{2.20cm}{2.20cm}{2.20cm}
\usepackage{times}
\usepackage{amsmath}
\usepackage{amssymb}
\usepackage{amsfonts}
\usepackage{xcolor}
\usepackage{hyperref}

\usepackage{amssymb,amsmath}
\usepackage{amssymb,latexsym}
\usepackage{amsfonts}
\usepackage{amssymb}
\usepackage{stmaryrd}
\usepackage{longtable}
\usepackage{xcolor}
\usepackage{amsmath,amssymb,amsfonts,amsthm,amscd,amstext,amsxtra,amsopn,array,url,mathrsfs,enumerate,anysize,enumitem}

\marginsize{2.20cm}{2.20cm}{2.20cm}{2.20cm}

\usepackage{amsrefs}
\newenvironment{abib}
  {\bibdiv\biblist\setupbib}
  {\endbiblist\endbibdiv}
  
  \def\setupbib{\catcode`@=\active}
\begingroup\lccode`~=`@
  \lowercase{\endgroup\def~}#1#{\gatherkey{#1}}
\def\gatherkey#1#2{\gatherkeyaux{#1}#2\gatherkeyaux}
\def\gatherkeyaux#1#2,#3\gatherkeyaux{\bib{#2}{#1}{#3}}

\numberwithin{equation}{section}

\newtheorem{theorem}{Theorem}[section]
\newtheorem{lemma}[theorem]{Lemma}

\newtheorem{proposition}[theorem]{Proposition}

\newtheorem*{exe}{Theorem 1.5}
\newtheorem{remarks}[theorem]{Remarks}
\newtheorem{remark}[theorem]{Remark}

\def\bals#1\nals{\begin{align*}#1\end{align*}}
\def\bal#1\nal{\begin{align}#1\end{align}}

\newcommand{\sqr}[2]{{\vcenter{\vbox{\hrule height#2pt
                \hbox{\vrule width#2pt height#1pt \kern#1pt
                \vrule width#2pt}\hrule height#2pt}}}}

\newcommand{\beq}{\begin{equation}}
\newcommand{\eeq}{\end{equation}}
\newcommand{\beqar}{\begin{eqnarray}}
\newcommand{\eeqar}{\end{eqnarray}}
\def\beqars{\begin{eqnarray*}}
\def\eeqars{\end{eqnarray*}}

\newcommand{\pmd}{\hspace{-3mm} \pmod}

\def \ds{\displaystyle}

\newcommand{\nn}{\mathbb{N}}
\newcommand{\zz}{\mathbb{Z}}

\newcommand{\hh}{\mathbb{H}}

\allowdisplaybreaks

\begin{document}

\markboth{A. Akbary \& Z. S. Aygin}
{Modular Equations}

\author{Amir Akbary}

\address{Department of Mathematics and Computer Science, University of Lethbridge, Lethbridge, AB T1K 3M4, Canada}%
\email{amir.akbary@uleth.ca}%

\author{Zafer Selcuk Aygin}

\address{Department of Mathematics and Statistics, University of Calgary, Calgary, AB T2N 1N4, Canada}%
\email{selcukaygin@gmail.com}%


\thanks{Research of the first author is partially supported by NSERC. Research of Z.S. Aygin is supported by a Pacific Institute for the Mathematical Sciences postdoctoral fellowship.}



\title{Sums of triangular numbers and sums of squares}






\maketitle

\begin{abstract}
For non-negative integers $a,b,$ and $n$, let $N(a, b; n)$ be the number of representations of $n$ as a sum of squares with coefficients $1$ or $3$ ($a$ of ones and $b$ of threes). Let $N^*(a,b; n)$ be the number of representations of $n$ as a sum of odd squares with coefficients $1$ or $3$ ($a$ of ones and $b$ of threes). We have that $N^*(a,b;8n+a+3b)$
is the number of representations of $n$ as a sum of triangular numbers with coefficients $1$ or $3$ ($a$ of ones and $b$ of threes).
It is known that for $a$ and $b$ satisfying $1\leq a+3b \leq 7$, we have
$$
 N^*(a,b;8n+a+3b)= \frac{2}{2+{a\choose4}+ab} N(a,b;8n+a+3b)
$$
and for $a$ and $b$ satisfying $a+3b=8$, we have
$$
N^*(a,b;8n+a+3b) = \frac{2}{2+{a\choose4}+ab}  \left( N(a,b;8n+a+3b) - N(a,b; (8n+a+3b)/4) \right).
$$
%
Such identities are not known for $a+3b>8$.  
In this paper, for general $a$ and $b$ with $a+b$ even, we prove asymptotic equivalence of formulas similar to the above, 
as $n\rightarrow\infty$.
One of our main results extends a theorem of Bateman, Datskovsky, and Knopp where the case $b=0$ and general $a$ was considered. Our approach is different from Bateman-Datskovsky-Knopp's proof where the circle method and singular series were used. We achieve our results by explicitly computing the Eisenstein components of
the generating functions of $N^*(a,b;8n+a+3b)$ and $N(a,b;8n+a+3b)$. The method we use is robust and can be adapted in studying the asymptotics of other representation numbers with general coefficients.
\end{abstract}


\section{Introduction and main results}

For $a,b,n \in \nn_0=\nn \cup \{0\}$, with $(a, b)\neq (0, 0)$, let 
\bals
N(a,b; n):=\# \{ (x_1,\ldots,x_a,y_1,\ldots, y_b) \in \zz^{a+b}; n=x_1^2 + \cdots + x_a^2 + 3y_1^2+ \cdots + 3y_b^2 \}
\nals
and
let $N^*(a, b; n)$ be the numbers of representations of $n$ by the quadratic form 
$$x_1^2 + \cdots+ x_a^2 + 3y_1^2+ \cdots + 3y_b^2$$
under the extra condition that $x_i$'s and $y_j$'s are odd integers. We let $N(a,b;n)$ and $N^*(a,b; n)=0$ if $n \not\in \nn_0$. Observe that
\begin{multline*}
N^*(a,b;8n+a+3b):= \# \left\{ (x_1,\ldots,x_a,y_1,\ldots, y_b) \in \zz^{a+b}; ~ \right.\\
 n=  \frac{x_1(x_1-1)}{2} + \cdots + \frac{x_a(x_a-1)}{2} 
\left. + 3\frac{y_1(y_1-1)}{2}+ \cdots + 3\frac{y_b(y_b-1)}{2} \right\}.
\end{multline*}
Thus, $N^*(a, b; 8n+a+3b)$ is the number of representations of $n$ as a sum of triangular numbers with coefficients $1$ or $3$ ($a$ of ones and $b$ of threes).  In this paper we are inspired by the existing results, in \cite{ ACH}, \cite{C4}, \cite{C7}, \cite{C12}, \cite{C10}, and \cite{SunPaper},  to examine the ratio of $N^*(a,b;8n+a+3b)$ to $N(a,b;8n+a+3b)$.

In \cite{C10}, Bateman and Knopp, using elementary observations and Jacobi's four square theorem, showed that for $1\leq a \leq 7$ we have
\bal
N^*(a,0;8n+a)= \frac{2}{2+{a\choose4}} N(a,0;8n+a). \label{eq31_1}
\nal
This result was rediscovered independently in 2002 by Barrucand, Cooper and Hirschhorn \cite{BCH} and proved by employing generating functions. A combinatorial proof is given in 2004 by Cooper and Hirschhorn in \cite{CH}. 

In \cite{C12}, Bateman, Datskovsky and Knopp revisited the identity \eqref{eq31_1} from a different point of view. From the Hardy-Littlewood circle method
it is known, for $3\leq a \leq 8$, that
\begin{equation}
\label{circle}
N(a, 0; n)=\frac{\pi^{a/2}}{\Gamma(a/2)} n^{a/2-1}\sum_{k=1}^{\infty} A_k^{(a)} (n),
\end{equation}
where the series in the above identity is the so called \emph{singular series} (see \cite[page 69]{C12})  for the definition of $A_k^{(a)} (n)$). A similar formula also holds for $N^*(a, 0; 8n+a+3b)$. By employing the multiplicativity of the singular series and computing the local component at $2$, in \cite{C12} it is shown that
\begin{equation}
\label{BDK}
\frac{N^*(a,0;8n+a)}{N(a,0;8n+a)}=\begin{cases} \frac{2}{2^{a-2} + 2^{(a-2)/2} \cos(\pi a/4)  + 1 }&if~a=3, 4, 5, 6, 7,\\ \frac{7\cdot 2^{3\nu}}{256\cdot 2^{3\nu}-480} &if~a=8, \end{cases}
\end{equation}
where $2^\nu$ is the largest power of $2$ dividing $8n+8$ (see \cite[Formulas (6.5) and (6.7)]{C12}). Observe that the formula in \eqref{BDK} for $3\leq a \leq 7$ coincides with \eqref{eq31_1}.
In fact, $$\frac{2}{2^{a-2} + 2^{(a-2)/2} \cos(\pi a/4)  + 1 }= \frac{2}{2+{a\choose4}},$$
for $1\leq a \leq  7$, however the left-hand side is strictly smaller than the right-hand side for $a\geq 8$. 
On the other hand the appearance of $\nu$ in \eqref{BDK},  for the case $a=8$, shows that the ratio ${N^*(8, 0; 8n+8)}/{N(8, 0; 8n+8)}$ is not a constant for all positive integer $n$. More generally, in \cite{C12} by employing the theory of modular forms of weight $a/2$ and appropriate multiplier system on the group $\Gamma_0(64)$, it is shown that the ratio ${N^*(a, 0; 8n+a)}/{N(a, 0; 8n+a)}$ is never constant for $a>7$, neither there exists a non-negative integer $M$ such that ${N^*(a, 0; 8n+a)}/{N(a, 0; 8n+a)}$ 
stays constant
for all $n\geq M$.
Interestingly though, for $a>8$, the asymptotic identity  
\begin{equation*}
N(a, 0; n)=\frac{\pi^{a/2}}{\Gamma(a/2)} n^{a/2-1}\sum_{k=1}^{\infty} A_k^{(a)} (n)+O(n^{a/4})
\end{equation*}
holds, where the main term is the same function as the right-hand side of \eqref{circle} (see \cite[formula (6.1)]{C12}). Now the computation of the local component of the singular series at prime $2$ reveals that for $a>8$ where $8\nmid a$, relations similar to \eqref{BDK} 
hold at the limit. 
More precisely, Theorem 3 in \cite{C12} states that, for each $a > 8$ where $a \not\equiv 0 \pmod{8}$,
\bal
\lim_{n \rightarrow \infty} \frac{N^*(a,0;8n+a)}{N(a,0;8n+a)} = \frac{2}{2^{a-2} + 2^{(a-2)/2} \cos(\pi a/4)  + 1 } \label{eq10_8}
\nal
and notes that
\bal
\lim_{n \rightarrow \infty} \frac{N^*(a,0;8n+a)}{N(a,0;8n+a)} \label{eq10_1}
\nal
does not exist when $a \equiv 0 \pmod{8}$.

In \cite{ACH}, Adiga, Cooper and Han considered the more general problem of the relation between representation numbers $N( {\lambda}; 8n+\lambda_1+\cdots+\lambda_m)$ and $N^*({\bf \lambda}; 8n+\lambda_1+\cdots+\lambda_m)$ in which ${\bf \lambda}=(\lambda_1, \cdots, \lambda_m)$ corresponds to a partition of a positive integer not exceeding $7$, $N({\bf \lambda}; n)$ denotes the number of representations of $n$ by the quadratic form $\lambda_1x_1^2+\cdots+\lambda_m x_m^2$, and $N^*({\bf \lambda}; n)$ is defined similarly with extra condition that $x_i$'s are odd. Their result in our special case  $\lambda_i\in \{1, 3\}$ states that for non-negative integers $a$ and $b$ satisfying $1\leq a+3b \leq 7$, we have
\bal
& N^*(a,b;{8n+a+3b}) = \frac{2}{2+{a\choose4}+ab} N(a,b;8n+a+3b).\label{eq31_3}
\nal

The cases corresponding to partitions of $8$ for $N^*$ and for a related function to $N$ (denoted by $\widetilde{N}$) were considered by Baruah, Cooper and Hirschhorn in \cite{C4}. 
Let $\widetilde{N}(a, b; n)$ be the number of representations of the integer $n$ as a sum of squares with coefficients $1$ or $3$ ($a$ of ones and $b$ of threes) and at least one odd component. We see that 
\begin{equation}
\label{n-tilden}
\widetilde{N}(a, b; 8n+a+3b)= N(a, b; 8n+a+3b)-N(a, b; (8n+a+3b)/4).
\end{equation}
Since 
$\widetilde{N}(a, b; 8n+a+3b)=N(a, b; 8n+a+3b)$ 
if $a+3b$ is not a multiple of $4$, we use the notation $\widetilde{N}(a, b; 8n+a+3b)$ only if $a+3b \equiv 0$ (mod~4).
In the special case $\lambda_i\in \{1, 3\}$, Theorem 1.4 of \cite{C4} implies that 
for non-negative integers $a$ and $b$ satisfying $a+3b=8$, we have
\bal
& N^*(a,b;{8n+a+3b}) = \frac{2}{2+{a\choose4}+ab}  \widetilde{N}(a, b; 8n+a+3b).
\label{eq31_4}
\nal

The above kind of exact identities between $N^*$ and $N$ or between $N^*$ and $\widetilde{N}$  are very rare. In fact, our first result in this paper proves that the above identities together with the ones listed in Lemma \ref{identities} are the only instances of such relations. 

\begin{theorem}
\label{alp_rels3-0}
Let $a, b\in \mathbb{N}_0$ be such that $a+b\equiv 0~({\rm mod}~2)$. The following statements hold.

\begin{itemize}
\item[(i)] We have
\begin{equation}
\label{NstarN}
N^*(a, b; 8n+a+3b)=
C_{a, b}
N(a, b; 8n+a+3b)
\end{equation}
for any $n\in \mathbb{N}_0$ and for a rational constant $C_{a, b}$, depending on $a$ and $b$,  if and only if 
$$(a, b)\in S=\{(0, 2), (0,4), (0, 6), (1, 1), (1, 3), (2, 0), (3, 1), (4, 0), (6, 0)\}.$$

\item[(ii)] If $a+3b\equiv 0$ (mod $4$), then
$${N^*}(a, b; 8n+a+3b)=
\widetilde{C}_{a, b}
{\widetilde{N}}(a, b; 8n+a+3b)$$
for any $n\in \mathbb{N}_0$ and for a rational $\widetilde{C}_{a, b}$, depending on $a$ and $b$, if and only if 
$$(a, b)\in \widetilde{S}=\{(0, 4), (0, 8), (1,1), (1, 5), (2, 2), (4, 0), (5, 1), (8, 0)\}.$$

\end{itemize}

\end{theorem}

%

In view of Theorem \ref{alp_rels3-0},
it is natural to study these representation functions asymptotically, for fixed integer parameters $a$ and $b$, as $n\rightarrow \infty$.
%
We prove the following.

\begin{theorem} \label{mainth} Let $1< a \in \nn$ and $b \in \nn_0$ be such that $a+b \equiv 0 \pmod{2}$ and assume $a+b \geq 4$. The following assertions hold:
\begin{itemize}
\item[(i)] If $a+3b \not\equiv 0 \pmod{8}$, then
\bal
&\lim_{n \rightarrow \infty}  \frac{N^*(a,b;{8n+a+3b})}{N(a,b;8n+a+3b) }= {\frac{ 2}{\ds 2^{a+b-2} + (-1)^b 2^{(a+b-2)/2}\cos (\pi (a+3b)/4)  +1}.} \label{eq10_6}
\nal
\item[(ii)] Let $a+3b \equiv 0 \pmod{8}$ and let $\nu\geq 3$ be the largest power of $2$ in the prime decomposition of $8n+a+3b$. Then, for fixed $\nu$, 
$$\lim_{\substack{n\rightarrow \infty \\ 2^\nu \Vert 8n+a+3b}}  \frac{N^*(a,b;8n+a+3b)}{N(a,b;8n+a+3b)}=  {\frac{ 2\delta(a, b, \nu)}{\ds 2^{a+b-2} + (-1)^b 2^{(a+b-2)/2} }, } $$
where
\begin{equation}
\label{Delta}
\delta(a, b, \nu):=\frac{\ds \sum_{j=\nu-1}^{\nu} (-1)^{bj}2^{j(a+b-2)/2}}{\ds -2+\sum_{j=0}^{\nu} (-1)^{bj} 2^{j(a+b-2)/2}}.
\end{equation}
\item[(iii)] If $a+3b \equiv 0 \pmod{4}$, then
\bal
\lim_{n \rightarrow \infty}  \frac{N^*(a,b;{8n+a+3b})}{\widetilde{N}(a,b;8n+a+3b) } = \frac{ 2}{\ds  2^{a+b-2} +(-1)^b 2^{(a+b-2)/2}\cos (\pi (a+3b)/4)}.
\label{eq10_4} 
\nal

%
%

\end{itemize}
\end{theorem}
%
\begin{remarks}
{\em We observe that  for non-negative integers $a$ and $b$ with $1\leq a+3b\leq 7$, the right-hand side of \eqref{eq10_6} is equal to $2/(2+{a\choose4}+ab)$. Thus, part (i) of Theorem \ref{mainth} provides an asymptotic extension of \eqref{eq31_3} for $a+b$ even and $a+3b \not\equiv 0 \pmod{8}$. In addition, it generalizes \eqref{eq10_8} for even $a$ and $b=0$ to general $a$ and $b$ with $a+b$ even.
Also,  for non-negative integers $a$ and $b$ with $a+3b=8$, the right-hand side of \eqref{eq10_4} is equal to $2/(2+{a\choose4}+ab)$. Thus,  part (iii) of Theorem \ref{mainth} provides an asymptotic extension of \eqref{eq31_4} for $a+b$ even and $a+3b \equiv 0 \pmod{4}$.
}

\end{remarks}

Let $a$ and $b$ be non-negative integers such that $a+b\geq 4$ is even. In this paper we also prove that the results described above for $4\leq a+3b \leq 8$ and the asymptotic results of Theorem \ref{mainth} are all consequence of a more general theorem (Theorem \ref{alp_rels3}) that can be described by modular interpretations of the generating functions of $N$, $N^*$, and $\widetilde{N}$.  In order to describe this general result we need some concepts and notations from the theory of modular forms (see Section \ref{modeqs} for references and information on modular forms).  Let $M_k(\Gamma_0(N), \chi)$ be the space of modular forms of weight $k$, level $N$, and character $\chi$. The following proposition demonstrates that the generating functions of $N(a, b; 8n+a+3b)$, $N^*(a, b; 8n+a+3b)$, and $\widetilde{N}(a, b; 8n+a+3b)$ are intimately related to functions in $M_k(\Gamma_0(N), \chi)$.

\begin{proposition}
\label{gen}
Let $a, b \in \mathbb{N}_0$ be such that $a+b \equiv 0~({\rm mod}~2)$. Let $q=e^{2\pi i z}$, where $z$ is in the upper half-plane. Then the following assertions hold.

(i) There are positive integers $k$, $N$, and a character $\chi$, all depending only on $a$ and $b$, such that 
\begin{equation}
\label{tns}
f_{N^*}(z):= \sum_{n=0}^{\infty} N^*(a, b; 8n+a+3b) q^{8n+a+3b}\in M_k(\Gamma_0(N), \chi).
\end{equation} 

(ii) Let 
\begin{equation*}
f_{N}(z):=\sum_ {\substack{ m=0 \\ m \equiv a+3b~({\rm mod}~8)}}^{\infty} N(a, b; m) q^{m}.
\end{equation*}
Then $f_{N}(z)\in M_k(\Gamma_0(N), \chi)$ for certain positive integers $k$, $N$, and a character $\chi$, all depending only on $a$ and $b$. Moreover, 
\begin{equation}
\label{tn}
\sum_{n=0}^{\infty} N(a, b; 8n+a+3b) q^{8n+a+3b}=f_N(z)-\sum_ {\substack{ \ell\in \mathbb{Z}\\0\leq 8\ell+a+3b<a+3b}} N(a, b; 8\ell+a+3b) q^{8\ell+a+3b}.
\end{equation}

(iii)  An analogous statement to (ii) holds for $\widetilde{N}$ by replacing $N$ with $\widetilde{N}$ in (ii).

\end{proposition}

For $f\in M_k(\Gamma_0(N), \chi)$, let $f^E$ be the Eisenstein component of $f$ (see Section \ref{modeqs} for the definition) and let $[n]f$ denote  the the $n$-th Fourier coefficient of the Fourier expansion of $f$ at $i\infty$. The following result unifies and generalizes the above known identities on $N$, $N^*$, and $\widetilde{N}$.

\begin{theorem}
\label{alp_rels3}
Let $a, b\in \mathbb{N}_0$ be such that $a+b\geq 4$ and $a+b\equiv 0~({\rm mod}~2)$. The following statements hold.
\begin{itemize}
\item[(i)] If $a+3b \not\equiv 0 \pmod{8}$, then, for all integers $n\geq 0$,
$$[8n+a+3b]f_{N^*}^E=
\frac{ 2}{\ds 2^{a+b-2} + (-1)^b 2^{(a+b-2)/2}\cos (\pi (a+3b)/4)  +1} [8n+a+3b]f_{N}^E.$$
%
\item[(ii)] If $a+3b \equiv 0 \pmod{8}$, then, for all integers $n\geq 0$ for which $2^\nu~\Vert~8n+a+3b$, 
$$[8n+a+3b]f_{N^*}^E=
\frac{ 2\delta(a, b, \nu)}{\ds 2^{a+b-2} + (-1)^b 2^{(a+b-2)/2}  } [8n+a+3b]f_{{N}}^E,$$
where $\delta(a, b, \nu)$ is given in \eqref{Delta}.
\item[(iii)] If $a+3b \equiv 0 \pmod{4}$, then, for all integers $n\geq 0$,
$$[8n+a+3b]f_{N^*}^E=
\frac{ 2}{\ds 2^{a+b-2} + (-1)^b 2^{(a+b-2)/2}\cos (\pi (a+3b)/4)  } [8n+a+3b]f_{\widetilde{N}}^E.$$
\end{itemize}

\end{theorem}

\begin{remark}
{\em Statements similar to Theorems \ref{mainth} and \ref{alp_rels3}
should hold with $3$ replaced by any other prime $p$. In fact, it appears that for an odd prime $p$ assertions resulting in replacing $3b$ with $pb$ throughout Theorems \ref{mainth} and \ref{alp_rels3} might hold. The tools for such a treatment are available at \cite{rmfpaper} and \cite{projections}. In this paper, we chose $p=3$ to demonstrate the techniques in a clear manner. On the other hand our methods will fail if $a+b$ is not even, because these cases concern half integer weight modular forms and in this realm finding analogues of Propositions \ref{thchi1}--\ref{thchi12} as explicit as we do in this work seems to be difficult.}
\end{remark}

We next describe the method we use in proving Theorems  \ref{alp_rels3} and \ref{mainth}.
An important ingredient of the proof of Theorem \ref{alp_rels3} 
is an explicit construction of the Eisenstein parts of the generating functions of  $N^*(a, b; 8n+a+3b)$  
and $N(a, b; n)$, based on a method devised and employed in \cite{rmfpaper}, \cite{sqfreepaper}, and \cite{projections}. In being able to compute the Eisenstein components explicitly,
we can explain here the underlying phenomenon using as little as possible technical details, see  Section \ref{modeqs} for details on Eisenstein series. 

Let us define
\bal
\label{thet}
\varphi(z):=\sum_{m =-\infty}^{\infty} q^{m^2},~  \Psi(z):=\sum_{m =1}^{\infty} q^{m(m-1)/2+1/8},\mbox{ and } \Psi_8(z):=\Psi(8z).
\nal
Then the generating functions for $N(a,b;n)$ and $N^*(a,b;8n+a+3b)$ can be given by
\bal
\label{thet-phi}
\sum_{n =0}^{\infty} N(a,b;n)q^n = \varphi^a(z) \varphi^b(3z)
\nal
and 
\bal
\label{thet-psi}
\sum_{n =0}^{\infty} N^*(a,b;8n+a+3b)q^{8n+a+3b} = 2^{a+b} {\Psi_8^a(z) \Psi_8^b(3z)},
\nal
respectively. 
In Section \ref{sectionfour} we use the theory of modular forms to write
\bals
 \Psi_8^a(z) \Psi_8^b(3z)= \sum_{n =0}^{\infty} \alpha_n q^n +  \sum_{n =0}^{\infty} \gamma_n q^n,\\
 \varphi^a(z) \varphi^b(3z)= \sum_{n =0}^{\infty} \beta_n q^n +  \sum_{n =0}^{\infty} \gamma_n' q^n,
\nals
where $\alpha_n$ and $\beta_n$ will be given explicitly in terms of the generalized divisor function $\sigma_{k}(\epsilon, \psi; n)$, associated to certain integer $k$ and certain Dirichlet characters $\epsilon$ and $\psi$, and
\begin{equation}
\label{bound-H}
\gamma_n=O_\epsilon(n^{(a+b-2)/4+\epsilon}) \mbox{ and } \gamma'_n=O_\epsilon(n^{(a+b-2)/4+\epsilon})
\end{equation}
for any $\epsilon>0$ (See Section \ref{sectionfour} for the exact definitions of $\alpha_n$, $\beta_n$, $\gamma_n$, and $\gamma^{\prime}_n$. Note that $\alpha_n=\gamma_n=0$ if $n\not\equiv a+3b$~(mod $8$)). 
In addition, 
we have
\begin{multline}
\label{relations}
[8n+a+3b]f_{N^*}^E=\alpha_{8n+a+3b},~[8n+a+3b]f_{N}^E=\beta_{8n+a+3b}/2^{a+b},\\
{\rm and}~~[8n+a+3b]f_{\widetilde{N}}^E=( \beta_{8n+a+3b} - \beta_{2n+(a+3b)/4}   )/2^{a+b}.
\end{multline}
Thus, considering these relations, we arrive in the following equivalent version of Theorem \ref{alp_rels3}.

\begin{exe}[Second Version] \label{alp_rels} Let $a,b \in \nn_0$ be such that $a+b \equiv 0 \pmod{2}$ and let $a+b \geq 4$. For all non-negative integers $n$,  the following assertions hold:
\begin{itemize}
\item[(i)] If $a+3b \not\equiv 0 \pmod{8}$, then we have
\bal
\alpha_{8n+a+3b} = \frac{ 2 \beta_{8n+a+3b} }{ 2^{a+b}\left(2^{a+b-2} + (-1)^b 2^{(a+b-2)/2} \cos (\pi (a+3b)/4)   +1\right)}. \label{alp_rels2}
\nal
\item[(ii)] If $a+3b \equiv 0 \pmod{8}$, then, for all integers $n\geq 0$ for which $2^\nu~\Vert~8n+a+3b$, 
\bal
\alpha_{8n+a+3b}=
\frac{ 2\delta(a, b, \nu)\beta_{8n+a+3b}}{\ds 2^{a+b} \left( 2^{a+b-2} + (-1)^b 2^{(a+b-2)/2}\right)},
\label{alp_rels4}
\nal
where $\delta(a, b, \nu)$ is given in \eqref{Delta}.

\item[(iii)] If $a+3b \equiv 0 \pmod{4}$, then we have
{\bal
\alpha_{8n+a+3b} = \frac{2(\beta_{8n+a+3b} - \beta_{2n+(a+3b)/4})}{{2^{a+b} \left(2^{a+b-2} + (-1)^b 2^{(a+b-2)/2}\cos (\pi (a+3b)/4)\right)}}. \label{alp_rels1}
\nal}
\end{itemize}
\end{exe}

The proof of Theorem \ref{alp_rels3} exploits the multiplicativity of the  generalized divisor functions in establishing \eqref{alp_rels2} and \eqref{alp_rels1}. 
%
We then show that if $a>1$, we have
\begin{equation}
\label{little-o}
n^{(a+b-1)/4}=o(\alpha_{8n+a+3b}),
\end{equation}
as $n\rightarrow\infty$.
Therefore, Theorem \ref{mainth} follows as a direct corollary of the relations
\begin{multline*}
N^*(a,b;8n+a+3b)=2^{a+b} (\alpha_{8n+a+3b}+\gamma_{8n+a+3b}),~N(a, b; 8n+a+3b)=\beta_{8n+a+3b}+\gamma_{8n+a+3b}^\prime, \\ {\rm and}~~ \widetilde{N}(a, b; 8n+a+3b)=(\beta_{8n+a+3b}-\beta_{2n+(a+3b)/4})+(\gamma_{8n+a+3b}^\prime-\gamma_{2n+(a+3b)/4}),
\end{multline*}
together with \eqref{bound-H},  \eqref{alp_rels2}, \eqref{alp_rels1}, and \eqref{little-o}.

\begin{remarks}
{\em (i) We exclude the cases $a=0$ or $a=1$ in Theorem \ref{mainth}, because in these cases $\alpha_{8n+a+3b}=0$ infinitely many times. However, one can deduce from \eqref{est1}--\eqref{est8} that, in these cases, $\alpha_{8n+a+3b}$ are always nonzero for certain subsequences of $(8n+a+3b)_{n \geq 0}$. Thus one can write an equivalent of Theorem \ref{mainth} in these cases for relevant subsequences of $(8n+a+3b)_{n \geq 0}$.

(ii) The proof of Theorem \ref{alp_rels3} uses modular identities of Propositions \ref{thchi1}--\ref{thchi12}. We can use the identities for $\psi_8^a(z)\psi_8^b(3z)$ 
to establish numerous formulas for representation numbers $N^*(a, b; n+a+3b)$. For example formulas corresponding to $(a, b)=(4, 0)$, $(8, 0)$ in \cite{ORW} and  $(2, 2)$
in \cite{cooperrmf} can be derived from Proposition \ref{thchi1}. The same is true for the main terms of the cases $(12, 0)$,  $(24, 0)$ in \cite{ORW} and the main term of $(2k, 2k)$ of \cite{cooperrmf}.
The main term of the case $(2k+1, 2k+1)$ in \cite{cooperrmf} is a consequence of 
Proposition \ref{thchi3}.
Also the case (6, 0) in \cite{ORW} can be derived from Proposition \ref{thchi4}  and the same for the main term of (10, 0) in \cite{ORW}.
(iii) The method exploited in this paper for the proof of Propositions \ref{thchi1}--\ref{thchi12} is described in a general setting in \cite{projections} and can be applied to a wide variety of representation problems.}
\end{remarks}

The structure of the paper is as follows. In Section \ref{modeqs} we review some basic facts from the theory of modular forms and give a proof of Proposition \ref{gen}. A proof of Theorem \ref{alp_rels3-0} is given in Section \ref{three}.
After establishing the explicit modular identities in Section \ref{sectionfour} and the proof of Theorem \ref{alp_rels3} in Section \ref{proof1_2}, the proof of Theorem \ref{mainth} will be described in Section \ref{proof1_1}.

\section{Modular forms} \label{modeqs}


For positive integer $N$, the level $N$ congruence subgroup $\Gamma_0(N)$ is 
\bals
\left\{ \begin{pmatrix}
a & b \\ c & d
\end{pmatrix} \Big\vert\ a,b,c,d \in \zz,\ c \equiv 0 \pmd{N},\ ad-bc=1 \right\}.
\nals
Let $k\in \nn$ and let $\chi$ be a Dirichlet character mod $N$, where $\chi(-1)=(-1)^k$.  We denote the space of modular forms of weight $k$ and character $\chi$ on $\Gamma_0(N)$ by $M_k(\Gamma_0(N),\chi)$.  (For $\chi$ induced from a primitive Dirichlet character mod $M$, where $M\mid N$, we sometimes write   $M_k((\Gamma_0(N),\psi)$ instead of $M_k(\Gamma_0(N),\chi)$.) To fix our notation, we set $\chi_D (*)= \left( \frac{D}{*} \right)$, where  $D$ is a fundamental discriminant and $\left( \frac{D}{*} \right)$ is the Kronecker symbol, thus $\chi_D$ is a primitive real Dirichlet character modulus $D$. Note that, $\chi_1$ is the principal character of modulus $1$.

A well-studied type of modular forms are certain quotients of the Dedekind eta function $\eta (z)$, which is a holomorphic function defined on the upper half plane $\hh$ by the product formula
\bals
\eta (z) = e^{\pi i z/12} \prod_{n=1}^{\infty} (1-e^{2\pi inz})= q^{1/24} \prod_{n=1}^{\infty} (1-q^n).
\nals
The quotients of products of $\eta(dz)$ for $d \in \nn$ are called eta quotients.  {The significance of eta quotients for us is that} 
\bals
\varphi(z)=\sum_{m =-\infty}^{\infty} q^{m^2}\mbox{ and }\Psi(z)=\sum_{m =1}^{\infty} q^{m(m-1)/2+1/8}
\nals
can be written as eta quotients. More precisely, from \cite[Theorem 3.5]{cooperbook}, we have
\begin{equation}
\label{etaquotient}
\varphi(z)= \frac{\eta^5(2z)}{\eta^2(z)\eta^2(4z)} \mbox{ and } \Psi(z)= \frac{\eta^2(2z)}{\eta(z)}.
\end{equation}
Set $\Psi_8(z):= \Psi(8z)$.
One can show that $\varphi^4(z)\in M_2(\Gamma_0(4), \chi_1)$ and $\Psi_8^{4}(z)$ $=\Psi^4(8z) \in  M_2(\Gamma_0(16), \chi_1)$, 
see \cite[Propositions 5.9.2 and 5.9.3]{cohenbook} for details.

The Eisenstein and cusp form subspaces of $M_k(\Gamma_0(N), \chi)$ are denoted by $E_k(\Gamma_0(N),\chi)$ and $S_k(\Gamma_0(N),\chi)$, respectively. Then we have
\bals
M_k(\Gamma_0(N),\chi) = E_k(\Gamma_0(N),\chi) \oplus S_k(\Gamma_0(N),\chi),
\nals
see \cite[p.\ 83]{stein}. Thus, for $f(z)=\sum_{n=0}^{\infty} a_f(n) q^n \in M_k(\Gamma_0(N), \chi)$, there are unique 
functions $E_f(z)=\sum_{n=0}^{\infty} e_f(n) q^n \in E_k(\Gamma_0(N), \chi)$ and 
$C_f(z)=\sum_{n=0}^{\infty} c_f(n) q^n \in S_k(\Gamma_0(N), \chi)$, such that 
$$a_f(n)=e_f(n)+c_f(n).$$
On the other hand, it is known that 
\begin{equation}
\label{Hecke Bound}
c_f(n)= O_\epsilon(n^{(k-1)/2+\epsilon}),
\end{equation}
for any $\epsilon>0$ (see \cite[p. 314]{cohenbook}). One  can use \cite[Theorem 1]{projections} to write $e_{\Psi_8^a(z) \Psi_8^b(3z)}(n)$ and $e_{\varphi^a(z) \varphi^b(3z)}(n)$ explicitly in terms of 
the generalized divisor functions defined by
\begin{align}
\sigma_{k}(\epsilon,\psi; n) :=\begin{cases} \sum_{1 \leq d\mid n}\epsilon(n/d){\psi(d)} d^{k} & \mbox{ if $n \in \nn$,} \\ 0 & \mbox{ if $n \not\in  \nn$,} \end{cases} \label{sodfnc}
\end{align}
for certain positive integer $k$ and certain primitive Dirichlet characters $\epsilon$ and $\psi$.

The space $E_k(\Gamma_0(N), \chi)$ admits a natural basis of weight $k$ Eisenstein series, which we describe now. Let $B_{k,\chi}$ be the Bernoulli number associated with the Dirichlet character $\chi$ of modulus $N$. Let $\epsilon$ and $\psi$ be two primitive Dirichlet characters of moduli $L$ and $M$, respectively, such that $\chi=\epsilon\cdot\psi$ (i.e., $\chi(n)=\epsilon(n) \psi(n)$ for all $n \in \zz$ such that $\gcd(n,N)=1$). Then the weight $k$ Eisenstein series of level $N$ and character $\chi$  
is defined as
\begin{align*}
{E_{k}(z; \epsilon,\psi)} :=& \epsilon(0) - \frac{2k}{B_{k,\chi}} \sum_{n =1}^{\infty} \sigma_{k-1}(\epsilon, \psi; n) e^{2\pi i n z}.
\end{align*}
It is known that when $k \geq 2$ and $(k,\chi) \neq (2,\chi_1)$ the collection $$\mathcal{E}_k(\Gamma_0(N), \chi)=\{ E_k(dz; \epsilon, \psi);~\epsilon \cdot \psi=\chi~{\rm and}~ LMd\mid N\}$$ forms a basis for the space $E_k(\Gamma_0(N), \chi)$, and {when $(k,\chi) = (2,\chi_1)$ the collection
\bals
\mathcal{E}_2(\Gamma_0(N), \chi_1) & =\{ E_2(dz; \epsilon, \psi);~\epsilon \cdot \psi=\chi_1,~(\epsilon,\psi) \neq (\chi_1,\chi_1) ~{\rm and}~ LMd\mid N\} \\
& \cup \{ E_2(z; \chi_1, \chi_1)-dE_2(dz; \chi_1, \chi_1);~ 1< d\mid N\}
\nals
forms a basis for the space $E_2(\Gamma_0(N), \chi_1)$}, see \cite[Theorem 5.9]{stein} for details. 

We next note that from $\eqref{etaquotient}$  for $a,b\in \mathbb{N}_0$, we have 
\bals
\Psi_8^a(z)\Psi_8^b(3z)=\frac{\eta^{2a}(16z)\eta^{2b}(48z)}{\eta^{a}(8z)\eta^{b}(24z)}
\nals
and 
\bals
\varphi^a(z)\varphi^b(3z)=\frac{\eta^{5a}(2z)\eta^{5b}(6z)}{\eta^{2a}(z)\eta^{2a}(4z)\eta^{2b}(3z)\eta^{2b}(12z)}.
\nals
In addition, letting $k=\lfloor (a+b)/4 \rfloor  \geq 1 $, by Proposition 5.9.2 and Proposition 5.9.3 of \cite{cohenbook}, we deduce that
\bal
\label{mod-psi}
\Psi_8^a(z)\Psi_8^b(3z) \in \begin{cases}
M_{2k}(\Gamma_0(48),\chi_1) & \mbox{if $a,b \equiv 0 \pmd{2}$ and $a+b \equiv 0 \pmd{4}$},\\ 
M_{2k+1}(\Gamma_0(48),\chi_{-3}) & \mbox{if $a,b \equiv 1 \pmd{2}$ and $a+b \equiv 2 \pmd{4}$},\\ 
M_{2k+1}(\Gamma_0(48),\chi_{-4}) & \mbox{if $a,b \equiv 0 \pmd{2}$ and $a+b \equiv 2 \pmd{4}$},\\ 
M_{2k}(\Gamma_0(48),\chi_{12}) & \mbox{if $a,b \equiv 1 \pmd{2}$ and $a+b \equiv 0 \pmd{4}$,}
\end{cases}
\nal
and
\bal
\label{mod-phi}
\varphi^a(z)\varphi^b(3z) \in \begin{cases}
M_{2k}(\Gamma_0(12),\chi_1) & \mbox{if $a,b \equiv 0 \pmd{2}$ and $a+b \equiv 0 \pmd{4}$},\\ 
M_{2k+1}(\Gamma_0(12),\chi_{-3}) & \mbox{if $a,b \equiv 1 \pmd{2}$ and $a+b \equiv 2 \pmd{4}$},\\ 
M_{2k+1}(\Gamma_0(12),\chi_{-4}) & \mbox{if $a,b \equiv 0 \pmd{2}$ and $a+b \equiv 2 \pmd{4}$},\\ 
M_{2k}(\Gamma_0(12),\chi_{12}) & \mbox{if $a,b \equiv 1 \pmd{2}$ and $a+b \equiv 0 \pmd{4}$}.
\end{cases}
\nal
We are now ready to prove Proposition \ref{gen}.

\begin{proof}[Proof of Proposition \ref{gen}] Part (i) is a consequence of \eqref{mod-psi}.
The identity \eqref{tn} in part (ii) is a direct corollary of the definition of $f_N(z)$. Now for $$g^{a, b}(z):=\varphi^a(z)\varphi^b(3z)= \sum_{n=0}^{\infty} N(a, b; n) q^n$$ and a Dirichlet character $\omega$ modulo $8$, let $$g_\omega^{a, b}(z)=  \sum_{n=0}^{\infty} N(a, b; n) \omega(n) q^n$$
be the twist of $g^{a, b}$ by $\omega$. From the orthogonality of characters we have that
$$f_N(z)= \sum_ {\substack{ m=0 \\ m \equiv a+3b~({\rm mod}~8)}}^{\infty} N(a, b; m) q^{m}=\frac{1}{\phi(8)} \sum_{\omega ({\rm mod}~8)}  {\bar{\omega}(a+3b)} g_{\omega}^{a, b}(z), $$
where $\phi$ is the Euler function. Since by \eqref{mod-phi} each $g^{a, b}$ is a modular form of weight $\kappa\in \{2k, 2k+1\}$, level $12$, and character $\chi\in \{\chi_1, \chi_{-3}, \chi_{-4}, \chi_{12}\}$, then by \cite[Proposition 10.3.18]{cohenbook} each $g_\omega^{a, b}$ is a modular form of weight $\kappa$, level $12 \times 8^2$, and character $\chi$ (Note that $\omega$'s are real and thus $\omega^2$'s are equal to the principal character modulo $2$).
From the above identity for $f_N(z)$ in terms of $g_{\omega}^{a, b}$'s, we conclude that $f_N(x)$ is a modular form of weight $\kappa$, level $12 \times 8^2$, and character $\chi$. This completes the proof of (ii). The proof of part (iii) is similar to part (ii) by observing that, by \eqref{mod-phi} and  \cite[Proposition 7.3.3 (b) and Theorem 7.3.4]{cohenbook}, $g^{a, b}(z)-g^{a, b}(4z)$ is a modular form of weight $\kappa$, level $12 \times 4$, and character $\chi$.

\end{proof}

\section{Proof of Theorem \ref{alp_rels3-0}}
\label{three}

We start by proving some lemmas which will be required in our proof.

\begin{lemma}
\label{identities}
(i) For all $n\in \mathbb{N}_0$, we have
$$N^*(1, 3; 8n+10)=\frac{2}{5} N(1, 3; 8n+10).$$
(ii) Let $a, b, n\in \mathbb{N}_0$ be such that $a+3b=4$, then 
$$N^*(a, b; 8n+a+3b)=\widetilde{N}(a, b; 8n+a+3b).$$
(iii) For all $n\in \mathbb{N}_0$, we have
$$N^*(1, 5; 8n+16)=\frac{1}{6} \widetilde{N}(1, 5; 8n+16).$$
\end{lemma}
\begin{proof}
(i) We observe that $N(1, 3; 2)=0$ and thus, by \eqref{tn}, the generating function of  $N(1, 3; 8n+10)$ is a modular form. More precisely, by \eqref{mod-psi} and the proof of part (ii) of Proposition \ref{gen}, $\sum_{n=0}^{\infty} N^*(1, 3; 8n+10) q^{8n+10}$ and $\sum_{n=0}^{\infty} N(1, 3; 8n+10) q^{8n+10}$ are both in the space $M_2(\Gamma_0(12 \times 8^2), \chi_{-3})$. Since this space is finite dimensional. The identity for all values of $n$ will be verified by establishing it for finitely many values of $n$, which can be done computationally in a straightforward manner. 

(ii) If $a+3b=4$, then $(a, b)=(0, 4), (1, 1),$ or $(4, 0)$. Since the square of any odd prime modulo $8$ is $1$ and square of any even prime modulo $8$ is $0$ or $4$, then the solutions of the equations $x_1^2+3x_2^2=8n+4$ and $x_1^2+x_2^2+x_3^2+x_4^2=8n+4$ have an odd component if and only if all their components are odd. This settles the identity for $(a, b)=(1, 1)$ and $(4, 0)$. Now observe that $N^*(0, 4; 8n+12)=\widetilde{N}(0, 4; 8n+12)=0$ if $n$ is not a multiple of $3$. For $n=3 n_1$, we observe that $N^*(0, 4; 8n+12)=N^*(4, 0; 8n_1+4)$ and  $\widetilde{N}(0, 4; 8n+12)=\widetilde{N}(4, 0; 8n_1+4)$, thus the result follows from the case (4, 0) that is already proved.

(iii) It follows by an argument identical to part (i) with the starting observation that $\widetilde{N}(1, 5, 0)=\widetilde{N}(1, 5, 8)=0$ and thus the generating function of $\widetilde{N}(1, 5; 8n+16)$ is a modular form.
\end{proof}


\begin{remark}
{\em In \cite[Section 4]{ACH} it is conjectured that if $k\geq 8$,  for any partition $\lambda=(\lambda_1, \cdots, \lambda_m)$ of $k$ with ${\rm gcd}(\lambda_1, \cdots, \lambda_m)=1$, the identity 
$N^*(\lambda; 8n+\lambda_1+\cdots+\lambda_m)=C_\lambda N(\lambda; 8n+\lambda_1+\cdots+\lambda_m)$
for all $n\in \mathbb{N}_0$ does not hold.
Here, $$C_\lambda= \frac{2}{2+{i_1 \choose 4}+{i_1 \choose 2}{i_2 \choose 1}+{i_1 \choose 1}{i_3 \choose 1}},$$
where $i_j$ denotes the number of parts in $\lambda$ which are equal to $j$. We note that part (i) of Lemma \ref{identities} provides a counterexample to this conjecture.}
\end{remark}

\begin{lemma}
\label{eight}
The following assertions hold.

(i) If for all $n\in \mathbb{N}_0$ we have $N^*(a, 0; 8n+a)=C_{a, 0} {N}(a, 0; 8n+a)$,  for an even integer $a\geq 2$ and some rational $C_{a, 0}$, then $a=2, 4,$ or $6$.

(ii) If for all $n\in \mathbb{N}_0$ we have $N^*(a, 0; 8n+a)=\widetilde{C}_{a, 0} \widetilde{N}(a, 0; 8n+a)$, for an integer $a\geq 4$ where $a\equiv 0$ (mod $4$) and some rational $C_{a, 0}$, then $a=4$ or $8$. 
\end{lemma}

\begin{proof}
We prove (ii), the proof of (i) is similar. Let $a>8$ be a multiple of 4 and assume that there is a rational constant $\widetilde{C}_{a, 0}$ such that $N^*(a, 0; 8n+a)=\widetilde{C}_{a, 0} \widetilde{N}(a, 0; 8n+a)$ for all $n\in \mathbb{N}_0$. Since $N^*(a, 0; 8n+a)>0$ for all $n$ we conclude that $\widetilde{C}_{a, 0}\neq 0$. The identity \eqref{n-tilden}, relating $\widetilde{N}$ to $N$, together with \eqref{tns} and \eqref{tn} imply that 
\begin{equation}
\label{above-0}
f_{N^*}(z)-\widetilde{C}_{a, 0} f_{\widetilde{N}}(z)=\widetilde{C}_{a, 0}\sum_ {\substack{ \ell \in \mathbb{Z}\\0\leq 8\ell+a<a}} \widetilde{N}(a, 0; 8\ell+a) q^{8\ell+a}.
\end{equation}
Now let $s$ be the least positive residue modulo $8$ of $a$. Since $a>8$, then $0< s\leq8$, and thus the right-hand side of \eqref{above-0} is a non-trivial finite exponential sum as $\widetilde{N}(a, 0; s)>0$ and $\widetilde{C}_{a, 0}\neq 0$. On the other hand by Proposition \ref{gen}, the left-hand side of \eqref{above-0} is a modular form. But this is a contradiction since a non-trivial  modular form has the real axis as a natural analytic boundary (see \cite[Lemma 3]{C12}) and thus cannot be a non-trivial finite exponential sum. 
\end{proof}

\begin{proof}[Proof of Theorem \ref{alp_rels3-0}] 
(i) 
Assume that \eqref{NstarN} holds. Thus, there is a rational constant  $C_{a, b}$ 
such that 
\begin{equation}
\label{equality}
N^*(a, b; 8n+a+3b)=C_{a, b} N(a, b; 8n+a+3b)
\end{equation}
for all $n\in \mathbb{N}_0$. Note that such $C_{a, b}$ is non-zero, since otherwise $N^*(a, b; 8n+a+3b)=0$, a contradiction. Then, from \eqref{equality} together with \eqref{tns} and \eqref{tn}, we deduce that
\begin{equation}
\label{above}
f_{N^*}(z)-C_{a, b} f_{N}(z)=C_{a, b}\sum_ {\substack{ \ell \in \mathbb{Z}\\0\leq 8\ell+a+3b<a+3b}} N(a, b; 8\ell+a+3b) q^{8\ell+a+3b}.
\end{equation}
We claim that if $(a, b)\not\in S$, then the left-hand side of \eqref{above} is a non-trivial finite exponential sum. We prove this by considering cases.

Case 1: If $a=0$ and $(a, b)\not\in S$, then \eqref{equality} holds trivially if $n$ is not a multiple of $3$ (both sides are zero in this case). For $n=3n_1$ we observe that $$N^*(0, b; 24n_1+3b)=N^*(b, 0; 8n_1+b)~~{\rm and}~~ N(0, b; 24n_1+3b)=N(b, 0; 8n_1+b).$$
But we know that if $a=0$ and $(a, b)\not\in S$, then, by Lemma \ref{eight} (i), $$N^*(b, 0; 8n_1+b)=C_{0, b}N(b, 0; 8n_1+b)$$
for all $n_1\in \mathbb{N}_0$ never holds.

Case 2: If $a=1$, $b=8k+3$ with $k\geq 1$, and $(a, b)\not\in S$, then 
$0<10=8(1)+2< 1+3b$ and $N(1, 3; 10)>0$. Thus under the given conditions the right-hand side of \eqref{above} is non-trivial. On the other hand, by Proposition \ref{gen}, the left-hand side of $\eqref{above}$ is a modular form. This is a contradiction as described in the proof of Lemma \ref{eight}.

Case 3: In all other cases (i.e., $a=1$, $b\neq 8k+3$ with $k\geq 1$, and $(a, b)\not\in S$, or  $a>1$ and $(a, b)\not\in S$) we have $N(a, b; s)>0$, where $0\leq s \leq 7$ is the least non-negative residue of $a+3b$ modulo $8$.
Now by an appeal to the modularity of the left-hand side of \eqref{above} and non-triviality of the right-hand side of \eqref{above}, and following an argument identical to the one described in the proof of Lemma \ref{eight} and in Case 2, we conclude that \eqref{NstarN} does not hold in these cases. 

On the other hand if $(a, b)\in S$, then the existence of $C_{a, b}$ and the identity \eqref{equality} follows from \cite[Lemma 2.7]{C10}, \cite[Theorem 1.2]{ACH}, and Lemma \ref{identities} (i).

(ii) The proof follows the argument of part (i), by replacing $N$ to $\widetilde{N}$, $S$ to $\widetilde{S}$, $8k+3$ to $8k+5$, and $s$ (the least non-negative residue of $a+3b$) to $s$ (the least positive residue of $a+3b$).  Observe that if $(a, b)\in \widetilde{S}$, then the existence of $\widetilde{C}_{a, b}$ in the identity $$N^*(a, b; 8n+a+3b)=\widetilde{C}_{a, b} \widetilde{N}(a, b; 8n+a+3b)$$ is a consequence of parts (ii) and (iii) of Lemma \ref{identities} and \cite[Theorem 1.4]{C4}.
\end{proof}

\section{Modular Identities}
\label{sectionfour}
We now state and prove the modular identities used in the proof of Theorem \ref{alp_rels}. We obtain these identities by expressing $ \Psi_{8}^{a}(z)\Psi_{8}^{b}(3z)$ (respectively $ \varphi^{a}(z)\varphi^{b}(3z)$) as an explicit linear combination of the Eisenstein series in the set $\mathcal{E}_m(\Gamma_0(N), \chi)$ and a cusp form in $S_m(\Gamma_0(N), \chi)$, where $m$, $N$, and $\chi$, as in \eqref{mod-psi} and \eqref{mod-phi}, are depending on the parity of $a$ and $b$ and classes of $a+b$ mod $4$. The identities concerning $ \Psi_{8}^{a}(z)\Psi_{8}^{b}(3z)$ are new, however formulas for $ \varphi^{a}(z)\varphi^{b}(3z)$ have appeared in previous works, {see \cite{rmfpaper}, \cite{projections} (for general $a$ and $b$), and \cite{cooperrmf} (for $a=b$).} For the sake of completeness we re-state the formulas for $ \varphi^{a}(z)\varphi^{b}(3z)$ here.

We first describe these identities in Propositions \ref{thchi1}--\ref{thchi12} and then we prove them.

\begin{proposition} \label{thchi1}
Let $a,b \in \nn_0$ {be} such that $a+b=4k  \geq 4 $ and both $a,b$ are even. Then, for some cusp form $C_1(z) \in S_{2k}(\Gamma_0(48),\chi_1)$, we have
\bals
\Psi_8^a(z) \Psi_8^b(3z) & = a_{1,4}   \left(((-3)^{a/2}-1) E_{2k}(4z; \chi_1, \chi_1) + (3^{2k}-(-3)^{a/2}) E_{2k}(12z; \chi_1, \chi_1) \right) \\
&  + a_{1,8}  \left(((-3)^{a/2}-1)E_{2k}(8z; \chi_1, \chi_1) + (3^{2k}-(-3)^{a/2}) E_{2k}(24z; \chi_1, \chi_1) \right) \\
&  + a_{1,16}  \left( ((-3)^{a/2}-1) E_{2k}(16z; \chi_1, \chi_1) + (3^{2k}-(-3)^{a/2}) E_{2k}(48z; \chi_1, \chi_1) \right)\\
&+ C_1(z),
\nals
where
\bals
a_{1,4}&= \frac{((-1)^{(a+3b)/4}  - 1) }{{ 2^{4k}}(2^{2k}-1)(3^{2k}-1)},\\
a_{1,8}&=  \frac{ ({2}^{2k} + 1 - (-1)^{(a+3b)/4}) }{{ 2^{4k}}(2^{2k}-1)(3^{2k}-1)},\\
a_{1,16}&= - \frac{ 2^{2k} }{{ 2^{4k}}(2^{2k}-1)(3^{2k}-1)}.
\nals
Also for some cusp form $C_2(z)\in S_{2k}(\Gamma_0(12),\chi_1)$, we have
\bals
 \varphi^{a}(z)\varphi^{b}(3z)& = b_{1,1}  \left( ((-3)^{a/2}-1) E_{2k}(z; \chi_1, \chi_1) + (3^{2k}-(-3)^{a/2}) E_{2k}(3z; \chi_1, \chi_1)  \right) \\
& + b_{1,2}  \left( ((-3)^{a/2}-1) E_{2k}(2z; \chi_1, \chi_1) + (3^{2k}-(-3)^{a/2}) E_{2k}(6z; \chi_1, \chi_1)  \right) \\
& + b_{1,4}  \left( ((-3)^{a/2}-1) E_{2k}(4z; \chi_1, \chi_1) + (3^{2k}-(-3)^{a/2}) E_{2k}(12z; \chi_1, \chi_1)  \right) \\
& + C_2(z),
\nals
where
\bals
b_{1,1}&= \frac{(-1)^{(a+3b)/4}}{(2^{2k}-1)(3^{2k}-1)} ,\\
b_{1,2}&= - \frac{ (1+(-1)^{(a+3b)/4})}{(2^{2k}-1)(3^{2k}-1)} , \\
b_{1,4}&= \frac{2^{2k}}{(2^{2k}-1)(3^{2k}-1)}.
\nals
\end{proposition}

\begin{proposition} \label{thchi3}
Let $a,b \in \nn_0$ {be} such that $a+b=4k+2 \geq 6 $ and both $a,b$ are odd. Then, for some cusp form $C_3(z)\in S_{2k+1}(\Gamma_0(48),\chi_{-3})$, we have
\bals
\Psi_8^a(z) \Psi_8^b(3z) & =   a_{2,4} \left( E_{2k+1}(4z;\chi_{1},\chi_{-3}) - (-1)^{(a+3b)/4} (-3)^{(a-1)/2} E_{2k+1}(4z;\chi_{-3},\chi_{1}) \right)\\
& + a_{2,8} \left( E_{2k+1}(8z;\chi_{1},\chi_{-3}) - (-3)^{(a-1)/2} E_{2k+1}(8z;\chi_{-3},\chi_{1}) \right)\\
& + a_{2,16} \left( E_{2k+1}(16z;\chi_{1},\chi_{-3}) + (-3)^{(a-1)/2} E_{2k+1}(16z;\chi_{-3},\chi_{1}) \right) + C_3(z),
\nals
where
\bals
a_{2,4}=& {\frac {1 -   \left( -1 \right) ^{(a+3b)/4}  }{{ 2^{4k+2}} ({2}^{2k+1}+1)}},\\
a_{2,8}=& {\frac { {2}^{2k+1} - \left( 1-  \left( -1 \right) ^{(a+3b)/4}  \right)}{{2^{4k+2}}(2^{2k+1}+1)}},\\
a_{2,16}=&-{\frac {1}{2^{2k+1} ({2}^{2k+1}+1)}}.
\nals
Also for some cusp form $C_4(z)\in S_{2k+1}(\Gamma_0(12),\chi_{-3})$, we have
\bals
 \varphi^{a}(z)\varphi^{b}(3z) & = b_{2,1} \left( E_{2k+1}(z;\chi_{1},\chi_{-3}) + (-3)^{(a-1)/2} E_{2k+1}(z;\chi_{-3},\chi_{1}) \right) \\
 & + b_{2,2} \left( E_{2k+1}(2z;\chi_{1},\chi_{-3}) - (-1)^{(a+3b)/4} (-3)^{(a-1)/2} E_{2k+1}(2z;\chi_{-3},\chi_{1}) \right)\\
 & + b_{2,4} \left( E_{2k+1}(4z;\chi_{1},\chi_{-3}) + (-3)^{(a-1)/2} E_{2k+1}(4z;\chi_{-3},\chi_{1}) \right) + C_4(z),
\nals
where
\bals
b_{2,1} & = -{\frac { \left( -1 \right) ^{(a+3b)/4}}{{2}^{2k+1}+1}},\\
b_{2,2} & = {\frac { 1 + (-1)^{(a+3b)/4}}{{2}^{2k+1}+1}}  ,\\
b_{2,4} & = {\frac {{2}^{2k+1}}{{2}^{2k+1}+1}}.
\nals
\end{proposition}

\begin{proposition} \label{thchi4}
Let $a,b \in \nn_0$ {be} such that $a+b=4k+2 \geq 6$ and both $a,b$ are even. Then, for some cusp form $C_5(z)\in S_{2k+1}(\Gamma_0(48),\chi_{-4})$, we have
\bals
\Psi_8^a(z) \Psi_8^b(3z) & =   a_{3,2} \left( E_{2k+1}(2z;\chi_{1},\chi_{-4}) +  (-1)^{(a+3b-2)/4} E_{2k+1}(2 z;\chi_{-4},\chi_{1}) \right)\\
& + a_{3,4} \left( E_{2k+1}(4z;\chi_{1},\chi_{-4}) + (-1)^{(a+3b-2)/4} 2^{2k} E_{2k+1}(4 z;\chi_{-4},\chi_{1}) \right) \\
& + a_{3,6} \left( E_{2k+1}(6z;\chi_{1},\chi_{-4}) - (-1)^{(a+3b-2)/4} E_{2k+1}(6 z;\chi_{-4},\chi_{1}) \right) \\
& + a_{3,12} \left( E_{2k+1}(12z;\chi_{1},\chi_{-4}) - (-1)^{(a+3b-2)/4} 2^{2k} E_{2k+1}(12 z;\chi_{-4},\chi_{1}) \right) \\
& + C_5(z),
\nals
where
\bals
a_{3,2}=&-\frac {(-3)^{a/2}-1}{{ 2^{4k+2}}(3^{2k+1}+1)},\\
a_{3,4}=& \frac {(-3)^{a/2}- 1}{{ 2^{4k+2}}(3^{2k+1}+1)},\\
a_{3,6}=& \frac { 3^{2k+1} + (-3)^{a/2} }{{ 2^{4k+2}}(3^{2k+1}+1)},\\
a_{3,12}= &- \frac {3^{2k+1}+(-3)^{a/2}}{{ 2^{4k+2}}(3^{2k+1}+1)}.
\nals
Also for some cusp form $C_6(z)\in S_{2k+1}(\Gamma_0(12),\chi_{-4})$, we have
\bals
\varphi^{a}(z)\varphi^{b}(3z) & = b_{3,1} \left( E_{2k+1}(z;\chi_{1},\chi_{-4}) + (-1)^{(a+3b-2)/4} 2^{2k} E_{2k+1}(z;\chi_{-4},\chi_{1}) \right)\\
& + b_{3,3} \left( E_{2k+1}(3z;\chi_{1},\chi_{-4}) - (-1)^{(a+3b-2)/4} 2^{2k} E_{2k+1}(3z;\chi_{-4},\chi_{1}) \right)+ C_6(z),
\nals
where
\bals
b_{3,1} & =- \frac {{(-3)}^{a/2}-1}{{3}^{2k+1}+1},\\
b_{3,3} & = \frac { {3}^{2k+1}+{(-3)}^{a/2} }{{3}^{2k+1}+1}.
\nals
\end{proposition}

\begin{proposition} \label{thchi12}
Let $a,b \in \nn_0$ {be} such that $a+b=4k \geq 4 $ and both $a,b$ are odd. Then, for some cusp form $C_7(z) \in M_{2k}(\Gamma_0(48),\chi_{12})$, we have
\bals
\Psi_8^a(z) \Psi_8^b(3z) & = 
 a_{4,2} \left( E_{2k}(2z; \chi_1,\chi_{12}) - (-3)^{(a-1)/2} E_{2k}(2z; \chi_{-3},\chi_{-4}) \right) \\
& +  a_{4,4} \left( E_{2k}(4z; \chi_1,\chi_{12}) + (-3)^{(a-1)/2} E_{2k}(4z; \chi_{-3},\chi_{-4}) \right) \\
&+ a_{5,2}  \left( E_{2k}(2z; \chi_{-4},\chi_{-3}) - (-3)^{(a-1)/2} E_{2k}(2z; \chi_{12},\chi_{1}) \right)\\
&+ a_{5,4}  \left( E_{2k}(4z; \chi_{-4},\chi_{-3}) + (-3)^{(a-1)/2}  E_{2k}(4z; \chi_{12},\chi_{1}) \right)
+ C_7(z),
\nals
where
\bals
 a_{4,2} & = \frac{1}{2^{4k}},\\
 a_{4,4}& = -\frac{1}{2^{4k}},\\
 a_{5,2} & =  \frac{( -1)^{(a+3b-2)/4}}{2^{4k}} ,\\
 a_{5,4} & =  \frac{(-1)^{(a+3b-2)/4}}{2^{2k+1}}.
\nals
Also for some cusp form $C_8(z) \in M_{2k}(\Gamma_0(12),\chi_{12})$, we have
\bals
 \varphi^{a}(z)\varphi^{b}(3z) & = b_{4,1} \left( E_{2k}(z; \chi_1,\chi_{12}) + (-3)^{(a-1)/2} E_{2k}(z; \chi_{-3},\chi_{-4}) \right)\\
 & + b_{5,1}  \left( E_{2k}(z; \chi_{-4},\chi_{-3}) + (-3)^{(a-1)/2} E_{2k}(z; \chi_{12},\chi_{1}) \right)+ C_8(z),
\nals
where
\bals
b_{4,1} & = 1,\\
b_{5,1} & = - (-1)^{(a+3b-2)/4} 2^{2k-1}.
\nals
\end{proposition}

\begin{proof}[Proofs of Propositions \ref{thchi1}--\ref{thchi12}] 

Propositions \ref{thchi1}--\ref{thchi12} {are} direct {consequences} of \cite[Theorem 1]{projections}. Alternatively, observe that since the constant coefficient of any cusp form vanishes at all cusps, in order to establish these modular identities, it would be enough to compute the constant coefficients of the expansions of $\Psi_8^a(z)\Psi_8^b(3z)$, $\varphi^a(z)\varphi^b(3z)$, and the Eisenstein series in $\mathcal{E}_m(N, \chi)$, for appropriate $m$, $N$, and $\chi$, at the related cusps.

A set of inequivalent cusps of $\Gamma_0(48)$ and $\Gamma_0(12)$ are given by
\bals
R(48)=\left\{ \frac{1}{1},\frac{1}{2},\frac{1}{3},\frac{1}{4},\frac{3}{4},\frac{1}{6},\frac{1}{8},\frac{1}{12},\frac{7}{12},\frac{1}{16},\frac{1}{24},\frac{1}{48} \right\}
\nals
and
\bals
R(12)=\left\{ \frac{1}{1},\frac{1}{2},\frac{1}{3},\frac{1}{4},\frac{1}{6},\frac{1}{12} \right\},
\nals
respectively, see \cite[Proposition 6.3.22]{cohenbook}. Letting $a+b$ to be even, the computation of constant coefficients of $\varphi^a(z) \varphi^b(3z)$ and $\Psi_8^a(z) \Psi_8^b(3z)$ are straightforward using \cite[Proposition 2.1]{Kohler}. These computations are carried out using the SAGE code provided at \cite[Appendix A]{projections}, and the simplified output is as follows, where $[f]_{a/c}$ denotes the constant term of $f$ at the cusp $a/c$:
\bals
& [\varphi^a(z) \varphi^b(3z)]_{1/1}= \left(\frac{-i}{2}\right)^{(a+b)/2}\left(\frac{1}{3}\right)^{b/2},\\
& [\varphi^a(z) \varphi^b(3z)]_{1/2}=0,\\
& [\varphi^a(z) \varphi^b(3z)]_{1/3}=i^a \left(\frac{-i}{2}\right)^{(a+b)/2},\\
& [\varphi^a(z) \varphi^b(3z)]_{1/4}=\left(\frac{1}{i \sqrt{3}}\right)^b,\\
& [\varphi^a(z) \varphi^b(3z)]_{1/6}=0,\\
& [\varphi^a(z) \varphi^b(3z)]_{1/12}=1,\\
& [\Psi_8^a(z) \Psi_8^b(3z)]_{1/1}= (-1)^{(a+b)/2} \left(\frac{1+i}{8}\right)^{a+b}\left(\frac{1}{\sqrt{3}}\right)^b,\\
& [\Psi_8^a(z) \Psi_8^b(3z)]_{1/2}= (-1)^{(a+b)/2} \left(\frac{1+i}{4 \sqrt{2}}\right)^{a+b}\left(\frac{1}{\sqrt{3}}\right)^b,\\
& [\Psi_8^a(z) \Psi_8^b(3z)]_{1/3}= (-1)^{(a+b)/2} \left(\frac{i-1}{8}\right)^a \left(\frac{i+1}{8}\right)^b,\\
& [\Psi_8^a(z) \Psi_8^b(3z)]_{1/4}= (-1)^{(a+3b)/2} \left(\frac{i-1}{4}\right)^{a+b}\left(\frac{1}{\sqrt{3}}\right)^b,\\
& [\Psi_8^a(z) \Psi_8^b(3z)]_{3/4}= (-1)^{(a+3b)/2} \left(\frac{i-1}{4}\right)^{a+b}\left(\frac{1}{\sqrt{3}}\right)^b,\\
& [\Psi_8^a(z) \Psi_8^b(3z)]_{1/6}= (-1)^{(a+3b)/2} \left(\sqrt{2}\right)^{a+b}\left(\frac{i-1}{8}\right)^a\left(\frac{i+1}{8}\right)^b,\\
& [\Psi_8^a(z) \Psi_8^b(3z)]_{1/8}= (-1)^{(a+3b)/2} \left(\frac{i}{2}\right)^a\left(\frac{1}{2\sqrt{3}}\right)^b,\\
& [\Psi_8^a(z) \Psi_8^b(3z)]_{1/12}= (-1)^{(a+b)/2} \left(\frac{i+1}{4}\right)^a\left(\frac{i-1}{4}\right)^b,\\
& [\Psi_8^a(z) \Psi_8^b(3z)]_{7/12}= (-1)^{(a+b)/2} \left(\frac{i+1}{4}\right)^a\left(\frac{i-1}{4}\right)^b,\\
& [\Psi_8^a(z) \Psi_8^b(3z)]_{1/16}=0,\\
& [\Psi_8^a(z) \Psi_8^b(3z)]_{1/24}= (-1)^{(a+b)/2} \left(\frac{i}{2}\right)^{a+b},\\
& [\Psi_8^a(z) \Psi_8^b(3z)]_{1/48}= 0.
\nals

We now, assuming that $a$ and $b$ are even and $a+b=4k $, define
\bals
C_1(z):=& \Psi_8^a(z) \Psi_8^b(3z) \\
& - a_{1,4}   \left(((-3)^{a/2}-1) E_{2k}(4z; \chi_1, \chi_1) + (3^{2k}-(-3)^{a/2}) E_{2k}(12z; \chi_1, \chi_1) \right) \\
&  - a_{1,8}  \left(((-3)^{a/2}-1)E_{2k}(8z; \chi_1, \chi_1) + (3^{2k}-(-3)^{a/2}) E_{2k}(24z; \chi_1, \chi_1) \right) \\
&  - a_{1,16}  \left( ((-3)^{a/2}-1) E_{2k}(16z; \chi_1, \chi_1) + (3^{2k}-(-3)^{a/2}) E_{2k}(48z; \chi_1, \chi_1) \right),
\nals
and
\bals
 C_2(z):= & \varphi^{a}(z)\varphi^{b}(3z)\\
 & - b_{1,1}  \left( ((-3)^{a/2}-1) E_{2k}(z; \chi_1, \chi_1) + (3^{2k}-(-3)^{a/2}) E_{2k}(3z; \chi_1, \chi_1)  \right) \\
& - b_{1,2}  \left( ((-3)^{a/2}-1) E_{2k}(2z; \chi_1, \chi_1) + (3^{2k}-(-3)^{a/2}) E_{2k}(6z; \chi_1, \chi_1)  \right) \\
& - b_{1,4}  \left( ((-3)^{a/2}-1) E_{2k}(4z; \chi_1, \chi_1) + (3^{2k}-(-3)^{a/2}) E_{2k}(12z; \chi_1, \chi_1)  \right).
\nals
We have
\begin{align*}
& [E_{2k}(dz; \chi_1,\chi_1)]_{a/c}  = \ds  \left(\frac{\gcd(c,d)}{d} \right)^{2k}  \mbox{ when $ k \neq 1$,}
\end{align*}
and
\begin{align*}
& [\left( E_{2}(z; \chi_1,\chi_1) -d E_{2}(dz; \chi_1,\chi_1)\right)]_{a/c} = 1 -  \frac{\gcd(c,d)^2}{d}
\end{align*}
(see \cite[Proposition 8.5.6]{cohenbook} and \cite[Formula (6.2)]{sqfreepaper}). Now it is straightforward, by using a symbolic computation software such as MAPLE, to check $[C_1(z)]_{a/c}=0$ for all $a/c \in R(48)$ and $[C_2(z)]_{a/c}=0$ for all $a/c \in R(12)$, i.e.\ $C_1(z)$ and $C_2(z)$ are cusp forms in the desired spaces. One can do similar verifications, by consulting \cite[Proposition 8.5.6]{cohenbook} and \cite[Formula (6.2)]{sqfreepaper} for the constant coefficients of Eisenstein series,  for the similarly defined functions  $C_3(z),$ $C_4(z)$, $C_5(z),$ $C_6(z)$ $C_7(z)$, and $C_8(z)$ corresponding to different parities of $a$ and $b$.
\end{proof}

\section{Proof of Theorem \ref{alp_rels3}} \label{proof1_2}

Throughout the section we let $\epsilon$ and $\psi$ to be real primitive Dirichlet characters with conductors $L$ and $M$, respectively, where $\gcd(L,M)=1$. The proofs we present in this section relies on some elementary {properties of} the sum of divisor functions defined in \eqref{sodfnc}. First, we note that $\sigma_{k}(\epsilon,\psi; n) $ is a multiplicative function of $n \in \nn$. Let $n= \prod_{p \mid n} p^{e_p}$ be the prime decomposition of $n$, then we have
\bal
\label{multi}
\sigma_{k}(\epsilon,\psi; n)  = \prod_{p \mid n} \sigma_{k}(\epsilon,\psi; p^{e_p}) = \prod_{p \mid n} \frac{(\psi(p) p^k)^{e_p+1} - \epsilon(p)^{e_p+1}}{\psi(p) p^k - \epsilon(p)}.
\nal

Next we note down an equation whose proof is straightforward using \eqref{sodfnc}.
\begin{lemma} \label{lemma3_2} Let $e, r \in \nn$. If $e \geq r$ we have
\bals
\sigma_{k} (\epsilon,\psi;2^{e}) - \epsilon^r(2) \sigma_{k} (\epsilon,\psi;2^{e-r}) = \sum_{i=0}^{r-1} \epsilon^i(2) (\psi(2) 2^k)^{e-i}(\psi(2) 2^k)^{e}.
\nals
\end{lemma}
We also record the following lemma which will be used later in the proof of Theorem \ref{alp_rels_2} and Lemma \ref{lem4_1}. Its proof is a direct consequence of \eqref{sodfnc} and the definition of $\chi_{-3}$ and $\chi_{-4}$.
\begin{lemma}
\label{chi-4}
(i) If $2\nmid n$, then
\bals
\sigma_{k} (\chi_{-4}, \chi_1; n)= \chi_{-4}(n) \sigma_{k}(\chi_1, \chi_{-4}; n),\\
\sigma_{k} (\chi_{-4}, \chi_{-3}; n)= \chi_{-4}(n) \sigma_{k}(\chi_1, \chi_{12}; n),
\nals
and
\bals
\sigma_{k} (\chi_{12}, \chi_{1}; n)= \chi_{-4}(n) \sigma_{k}(\chi_{-3}, \chi_{-4}; n).
\nals
(ii) If $3\nmid n$, then
\bals
\sigma_{k} (\chi_{-3}, \chi_1; n)= \chi_{-3}(n) \sigma_{k}(\chi_1, \chi_{-3}; n)
\nals
and
\bals
\sigma_{k} (\chi_{-3}, \chi_{-4}; n)= \chi_{-3}(n) \sigma_{k}(\chi_1, \chi_{12}; n).
\nals
\end{lemma}
\begin{proof}
We prove the first identity, the proofs for the rest are similar. First of all note that if $2\nmid n$, then $\chi_{-4}(n/d)=\chi_{-4}(n) \chi_{-4}(d)$. Thus,
\begin{multline*}
\sigma_{k}(\chi_{-4}, \chi_1; n)= \sum_{d\mid n} \chi_{-4}(n/d) \chi_1(d) d^{k}\\
= \chi_{-4}(n) \sum_{d\mid n} \chi_1(n/d) \chi_{-4}(d) d^{k}=\chi_{-4}(n) \sigma_k(\chi_1, \chi_{-4}; n).
\end{multline*}
\end{proof}

In the remainder of this section we prove Theorem \ref{alp_rels3}. Recall that
\begin{multline}
\label{psiab}
 \Psi_8^a(z) \Psi_8^b(3z)= \frac{1}{2^{a+b}} \sum_{n=0}^{\infty} N^*(a, b; 8n+a+3b) q^{8n+a+3b}\\= \sum_{n =0}^{\infty} \alpha_{8n+a+3b} q^{8n+a+3b} +  \sum_{n =0}^{\infty} \gamma_{8n+a+3b} q^{8n+a+3b},
\end{multline}
 where the sum involving $\alpha_{8n+a+3b}$'s is the Eisenstein part and the one with $\gamma_{8n+a+3b}$'s is the cusp part of  $\Psi_8^a(z) \Psi_8^b(3z)$ as they are given in Propositions \ref{thchi1}--\ref{thchi12}. Similarly
\bal
\label{phiab}
\varphi^a(z) \varphi^b(3z)=\sum_{n =0}^{\infty} N(a,b;n) q^n = \sum_{n =0}^{\infty} \beta_n q^n +  \sum_{n =0}^{\infty} \gamma_n' q^n
\nal
as they are given in Propositions \ref{thchi1}--\ref{thchi12}.

The relations between $\alpha_n$ and $\beta_n$ given in Theorem \ref{alp_rels_2} below imply the relations given in Theorem \ref{alp_rels3}. The implication will be proven at the end of this section.

\begin{theorem} \label{alp_rels_2} Let $a,b \in \nn_0$ and let $a+b \geq 4$. Then, for all $n \in \nn_0$, the following hold.
\begin{itemize}
\item[(i)] If $a+b \equiv 2 \pmod{4}$ and {$a,b$ are both} even, or $a+b \equiv 0 \pmod{4}$ and $a,b$ are both odd, then we have
\bal
\alpha_{8n+a+3b} = \frac{ 2 \beta_{8n+a+3b} }{ 2^{a+b}(2^{a+b-2}+1)}. \label{alp_rels2_2}
\nal
\item[(ii)] If $a+b \equiv 0 \pmod{4}$ and {$a,b$ are both} even, or $a+b \equiv 2 \pmod{4}$ and {$a,b$ are both} odd, then we have
\bal
\alpha_{8n+a+3b} = \frac{2(\beta_{8n+a+3b} -  \beta_{2n+(a+3b)/4})}{2^{a+b} \left(2^{a+b-2} + (-1)^b 2^{(a+b-2)/2}\cos{\left(\pi(a+3b)/4\right)} \right) }. \label{alp_rels1_1}
\nal
\end{itemize}
\end{theorem}

\begin{proof}
We start by proving \eqref{alp_rels1_1} when $a+b \equiv 0 \pmod{4}$ and $a,b$ are both even. In this case $4 \mid a+3b$. Therefore if we let $8n+a+3b = 2^{e_2} n_2$ with $\gcd(n_2,2)=1$, then $e_2 \geq 2$. The cases $e_2=2$, $e_2=3$ and $e_2 \geq 4$ needs to be treated individually. We only give the details of the case when $e_2 \geq 4$, as the remaining two cases can be handled similarly. We start the proof by noting that when $e_2 \geq 4$, we have
\begin{equation}
\label{recall}
(-1)^{(a+3b)/4} = 1.
\end{equation}
Hence, in Proposition \ref{thchi1},   $a_{1,4}=0$. Moreover, by employing Lemma \ref{lemma3_2} for $r=4$ and the fact that
\bals
a_{1,8}+a_{1,16}=0
\nals
in Proposition \ref{thchi1} we obtain
\bal
\frac{B_{2k}}{-4k}\alpha_{8n+a+3b} & = \left( a_{1,8} \sigma_{2k-1}(\chi_1,\chi_1; 2^{e_2-3}) + a_{1,16} \sigma_{2k-1}(\chi_1,\chi_1; 2^{e_2-4})  \right) \nonumber \\
& \times \left(((-3)^{a/2}-1) \sigma_{2k-1}(\chi_1,\chi_1; n_2) + (3^{2k}-(-3)^{a/2}) \sigma_{2k-1}(\chi_1,\chi_1;n_2/3) \right), \nonumber 
\nal
which, by \eqref{multi}, is
\bal
 & = a_{1,8}  (2^{2k-1})^{e_2-3}  \nonumber \\
& \times  \left(((-3)^{a/2}-1) \sigma_{2k-1}(\chi_1,\chi_1; n_2) + (3^{2k}-(-3)^{a/2}) \sigma_{2k-1}(\chi_1,\chi_1;n_2/3) \right), \label{eq3_1}
\nal
and
\bal
& \frac{B_{2k}}{-4k}(\beta_{8n+a+3b} - \beta_{2n+(a+3b)/4} ) \nonumber \\
& =\left( b_{1,1}  (\sigma_{2k-1}(\chi_1,\chi_1; 2^{e_2}) - \sigma_{2k-1}(\chi_1,\chi_1; 2^{e_2-2})) \right.  \nonumber \\
& + b_{1,2} ( \sigma_{2k-1}(\chi_1,\chi_1; 2^{e_2-1}) - \sigma_{2k-1}(\chi_1,\chi_1; 2^{e_2-3}))  \nonumber \\
& + \left. b_{1,4} (\sigma_{2k-1}(\chi_1,\chi_1; 2^{e_2-2}) - \sigma_{2k-1}(\chi_1,\chi_1; 2^{e_2-4}))  \right) \nonumber \\
& \times \left(((-3)^{a/2}-1) \sigma_{2k-1}(\chi_1,\chi_1; n_2) + (3^{2k}-(-3)^{a/2}) \sigma_{2k-1}(\chi_1,\chi_1;n_2/3) \right),  \nonumber 
\nal
which, by \eqref{multi}, is
\bal
& = (2^{2k-1} + 1) \left( b_{1,1} 2^{4k-2} + b_{1,2} 2^{2k-1} +  b_{1,4}   \right) (2^{2k-1})^{e_2-3}  \nonumber \\
& \times \left(((-3)^{a/2}-1) \sigma_{2k-1}(\chi_1,\chi_1; n_2) + (3^{2k}-(-3)^{a/2}) \sigma_{2k-1}(\chi_1,\chi_1;n_2/3) \right). \label{eq3_2}
\nal
Further, using \eqref{recall}, we observe that 
\bal
(2^{2k-1} + 1) \left( b_{1,1} 2^{4k-2} + b_{1,2} 2^{2k-1} +  b_{1,4}   \right) & = a_{1,8} 2^{6k-2}(2^{2k-1}+1).  \label{eq3_3}
\nal
Now, \eqref{alp_rels1_1} follows by combining \eqref{eq3_1}--\eqref{eq3_3} and noting that $k=(a+b)/4$.

The proof when $a+b \equiv 2 \pmod{4}$ and $a,b$ are both odd is similar to the above proof by employing Proposition \ref{thchi3}, so we skip it here.

We next prove \eqref{alp_rels2_2} when $a,b$ are both odd and 
$a+b \equiv 0 \pmod{4}$.
In this case $2~\Vert~ a+3b$. Therefore we have $8n+a+3b = 2 n_2$ with $\gcd(n_2,2)=1$. 
Thus, $\chi_{-4}(n_2)=(-1)^{(a+3b-2)/4}$ and so by Lemma \ref{chi-4} (i) we have 
\bal
\sigma_{2k-1}( \chi_{-4},\chi_{-3}; n_2)= (-1)^{(a+3b-2)/4}\sigma_{2k-1}( \chi_1,\chi_{12}; n_2) \label{bal1}
\nal
and
\bal
\sigma_{2k-1}( \chi_{12},\chi_{1}; n_2)= (-1)^{(a+3b-2)/4}\sigma_{2k-1}( \chi_{-3},\chi_{-4}; n_2).\label{bal2}
\nal
These identities together with Proposition \ref{thchi12} yield
\bal
\frac{B_{2k,\chi_{-12}}}{-4k} \alpha_{8n+a+3b} & = a_{4,2} \left( \sigma_{2k-1}(\chi_1,\chi_{12}; n_2) - (-3)^{(a-1)/2} \sigma_{2k-1}(\chi_{-3},\chi_{-4}); n_2 \right)  \nonumber \\
& \quad + a_{5,2}  \left( \sigma_{2k-1}(\chi_{-4},\chi_{-3}; n_2) - (-3)^{(a-1)/2} \sigma_{2k}(\chi_{12},\chi_{1};n_2) \right) \nonumber \\
& = \left( a_{4,2}  + a_{5,2} (-1)^{(a+3b-2)/4} \right)  \nonumber \\
&  \quad \times \left( \sigma_{2k-1}( \chi_1,\chi_{12}; n_2) -  (-3)^{(a-1)/2} \sigma_{2k-1}(\chi_{-3},\chi_{-4}; n_2) \right),  \nonumber 
\nal 
which is
\bal
& = \left( \frac{1}{2^{4k-1}} \right) \left( \sigma_{2k-1}( \chi_1,\chi_{12}; n_2) -  (-3)^{(a-1)/2} \sigma_{2k-1}(\chi_{-3},\chi_{-4}; n_2) \right). \label{eq3_9}
\nal
Next we use \eqref{multi} and employ the values
\bals
& \sigma_{2k-1}(\chi_1,\chi_{12}; 2) =1,\\
& \sigma_{2k-1}(\chi_{-3},\chi_{-4}; 2)=-1,\\
& \sigma_{2k-1}(\chi_{-4},\chi_{-3}; 2 )=-2^{2k-1},\\
& \sigma_{2k-1}(\chi_{12},\chi_{1}; 2)=2^{2k-1},
\nals
in Proposition \ref{thchi12} to obtain
\bal
\frac{B_{2k,\chi_{-12}}}{-4k} \beta_{8n+a+3b} & = b_{4,1} \left( \sigma_{2k-1}(\chi_1,\chi_{12}; 2 n_2) + (-3)^{(a-1)/2} \sigma_{2k-1}(\chi_{-3},\chi_{-4}; 2 n_2) \right) \nonumber \\
 & \quad  + b_{5,1}  \left( \sigma_{2k-1}(\chi_{-4},\chi_{-3}; 2 n_2) + (-3)^{(a-1)/2} \sigma_{2k-1}(\chi_{12},\chi_{1}; 2 n_2) \right) \nonumber \\
& = b_{4,1} \left(  \sigma_{2k-1}(\chi_1,\chi_{12}; n_2) - (-3)^{(a-1)/2} \sigma_{2k-1}(\chi_{-3},\chi_{-4}; n_2) \right) \nonumber \\
& \quad  - b_{5,1}  2^{2k-1} \left( \sigma_{2k-1}(\chi_{-4},\chi_{-3}; n_2) - (-3)^{(a-1)/2} \sigma_{2k-1}(\chi_{12},\chi_{1}; n_2) \right), \nonumber 
\nal
which, by \eqref{bal1} and \eqref{bal2}, is
\bal
& = \left( b_{4,1}  - b_{5,1} 2^{2k-1} (-1)^{(a+3b-2)/4} \right)  \nonumber \\
& \quad \times \left( \sigma_{2k-1}( \chi_1,\chi_{12}; n_2) - (-3)^{(a-1)/2} \sigma_{2k-1}( \chi_{-3},\chi_{-4}; n_2) \right), \nonumber
\nal
where, by Proposition \ref{thchi12}, is 
\bal
& = \left( 2^{4k-2} + 1 \right) \left( \sigma_{2k-1}( \chi_1,\chi_{12}; n_2) - (-3)^{(a-1)/2} \sigma_{2k-1}( \chi_{-3},\chi_{-4}; n_2) \right). \label{eq3_10}
\nal
Now we combine \eqref{eq3_9} and \eqref{eq3_10} to obtain \eqref{alp_rels2_2} in this case.

The proof of \eqref{alp_rels2_2} when $a,b$ are both even and $ a+b\equiv 2 \pmod{4}$, via Proposition \ref{thchi4}, is similar to the case just discussed.
\end{proof}

The next lemma will enable us to simplify \eqref{alp_rels1_1} in certain cases.

\begin{lemma} \label{lem_b} 
Let $a,b \in \nn$ { be} such that $a+b \equiv 0 \pmod{2}$ and $4\mid a+3b$. The following statements hold.
\begin{itemize}
\item[(i)] If $a+3b \equiv 4 \pmod{8}$ then, for all $n\in \mathbb{N}_0$, we have
\bals
\beta_{(8n+a+3b)/4}= \frac{ \beta_{8n+a+3b}}{{{ 2^{a+b-2}} +(-1)^b 2^{(a+b-2)/2} \cos{\left(\pi(a+3b)/4\right)} + 1 }}.
\nals
\item[(ii)] If $a+3b \equiv 0 \pmod{8}$ and $2^\nu~\Vert~8n+a+3b$, then
\bals
\beta_{(8n+a+3b)/4}=\frac{\ds -2+\sum_{j=0}^{\nu-2} (-1)^{bj}2^{j(a+b-2)/2}}{\ds -2+\sum_{j=0}^{\nu} (-1)^{bj} 2^{j(a+b-2)/2}} \beta_{8n+a+3b}.
\nals

\end{itemize}

\end{lemma}

\begin{proof} (i) Let $a,b \in \nn$ { be} such that $a+b \equiv 0 \pmod{2}$. Then $a+3b \equiv 4 \pmod{8}$ only if either $a+b \equiv 0 \pmod{4}$ and $a,b \equiv 0 \pmod{2}$, or $a+b \equiv 2 \pmod{4}$ and $a,b \equiv 1 \pmod{2}$. We write $8n+a+3b=4 n_2$, where $\gcd(2,n_2)=1$. Note that
\bal
(-1)^{(a+3b)/4}=-1.
\label{minus1}
\nal

We consider the case that $a+b \equiv 2 \pmod{4}$ and $a,b \equiv 1 \pmod{2}$, the proof in the other case is similar (using  Proposition \ref{thchi1}). Note that, by applying \eqref{minus1} in Proposition \ref{thchi3}, 
\bals
& \frac{B_{2k+1,\chi_{-3}}}{-(4k+2)} \beta_{8n+a+3b} \\
& = b_{2,1} \left( \sigma_{2k}(\chi_{1},\chi_{-3};4n_2) + (-3)^{(a-1)/2} \sigma_{2k}(\chi_{-3},\chi_{1});4n_2 \right) \\
 & + b_{2,2} \left( \sigma_{2k}(\chi_{1},\chi_{-3};2n_2) + (-3)^{(a-1)/2} \sigma_{2k}(\chi_{-3},\chi_{1};2n_2) \right)\\
 & + b_{2,4} \left( \sigma_{2k}(\chi_{1},\chi_{-3};n_2) + (-3)^{(a-1)/2} \sigma_{2k}(\chi_{-3},\chi_{1};n_2) \right),
\nals
where, by \eqref{multi}, is
\bals
& = \left( b_{2,1} ({ 2^{4k}}-2^{2k}+1) + b_{2,2} (1-2^{2k}) + b_{2,4} \right) \\
& \times \left(   \sigma_{2k}(\chi_{1},\chi_{-3};n_2) + (-3)^{(a-1)/2} \sigma_{2k}(\chi_{-3},\chi_{1});n_2 \right),
\nals
and, by \eqref{sodfnc},
\bals
& \frac{B_{2k+1,\chi_{-3}}}{-(4k+2)} \beta_{(8n+a+3b)/4} \\
& = b_{2,1} \left( \sigma_{2k}(\chi_{1},\chi_{-3};n_2) + (-3)^{(a-1)/2} \sigma_{2k}(\chi_{-3},\chi_{1});n_2 \right).
\nals

Hence we have
\bals
\beta_{(8n+a+3b)/4}& = \frac{b_{2,1}}{b_{2,1} ({ 2^{4k}}-2^{2k}+1) + b_{2,2} (1-2^{2k}) + b_{2,4} } \beta_{8n+a+3b}\\
& = \frac{\beta_{8n+a+3b}}{ { 2^{4k}} + 2^{2k} + 1 },
\nals
where $k=(a+b-2)/4$.

(ii) We note that in view  of $a+b \equiv 0 \pmod{2}$, we have $a+3b \equiv 0 \pmod{8}$ only if either $a+b \equiv 0 \pmod{4}$ and $a,b \equiv 0 \pmod{2}$, or $a+b \equiv 2 \pmod{4}$ and $a,b \equiv 1 \pmod{2}$. We write $8n+a+3b=2^\nu n_2$, where $\nu\geq 3$ and $\gcd(2,n_2)=1$. Observe that 
\bal
(-1)^{(a+3b)/4}=1.
\label{plus1}
\nal
We consider the case that $a+b \equiv 2 \pmod{4}$ and $a,b \equiv 1 \pmod{2}$, the proof in the other case is similar (using Proposition \ref{thchi1}). Now by employing \eqref{plus1} in Proposition \ref{thchi3} and steps similar to part (i) we get 
\bals
& \frac{B_{2k+1,\chi_{-3}}}{-(4k+2)} \beta_{8n+a+3b} \\
&= \left( b_{2,1} \sum_{j=0}^{\nu} (-1)^j (2^{2k})^j + b_{2,2}  \sum_{j=0}^{\nu-1} (-1)^j (2^{2k})^j + b_{2,4} \sum_{j=0}^{\nu-2} (-1)^j (2^{2k})^j  \right) \\
& \times \left(   \sigma_{2k}(\chi_{1},\chi_{-3};n_2) - (-3)^{(a-1)/2} \sigma_{2k}(\chi_{-3},\chi_{1});n_2 \right),
\nals
and 
\bals
& \frac{B_{2k+1,\chi_{-3}}}{-(4k+2)} \beta_{(8n+a+3b)/4} \\
&= \left( b_{2,1} \sum_{j=0}^{\nu-2} (-1)^j (2^{2k})^j + b_{2,2}  \sum_{j=0}^{\nu-3} (-1)^j (2^{2k})^j + b_{2,4} \sum_{j=0}^{\nu-4} (-1)^j (2^{2k})^j  \right) \\
& \times \left(   \sigma_{2k}(\chi_{1},\chi_{-3};n_2) - (-3)^{(a-1)/2} \sigma_{2k}(\chi_{-3},\chi_{1});n_2 \right),
\nals
where $\sum_{j=0}^{\nu-4}$ is zero if $\nu=3$. Now the result follows by comparing the above two formulas using the values of $b_{2, 1}$, $b_{2, 2}$, and $b_{2, 4}$  given in Proposition \ref{thchi3}. Note that $k=(a+b-2)/4$.
\end{proof}

\begin{proof}[Proof of Theorem \ref{alp_rels3} (Second Version)] 
(i) For even $a+b$, we consider cases $a+ 3b \equiv 2 \pmod{4}$ and  $a+3b \equiv 4\pmod{8}$ separately.

The condition $a+3b \equiv 2 \pmod{4}$ is possible only if either $a+b \equiv 2 \pmod{4}$ and $a,b \equiv 0 \pmod{2}$, or $a+b \equiv 0 \pmod{4}$ and $a,b \equiv 1 \pmod{2}$. Then the result in part (i) of Theorem \ref{alp_rels3} in this case follows from Theorem \ref{alp_rels_2} (i) and $\cos(\pi (a+3b)/4)=0$.

On the other hand $a+3b \equiv 4 \pmod{8}$ is possible only if either $a+b \equiv 0 \pmod{4}$ and $a,b \equiv 0 \pmod{2}$, or $a+b \equiv 2 \pmod{4}$ and $a,b \equiv 1 \pmod{2}$. Then the result in part (i) of Theorem \ref{alp_rels3}  in this case follows by employing Lemma \ref{lem_b} (i)
in Theorem \ref{alp_rels_2} (ii). 

(ii) The condition $a+3b \equiv 0 \pmod{8}$ is possible only if either $a+b \equiv 0 \pmod{4}$ and $a,b \equiv 0 \pmod{2}$, or $a+b \equiv 2 \pmod{4}$ and $a,b \equiv 1 \pmod{2}$. Then the result in part (ii) of Theorem \ref{alp_rels3}  follows by employing Lemma \ref{lem_b} (ii)
in Theorem \ref{alp_rels_2} (ii).

(iii) We note that $a+3b \equiv 0 \pmod{4}$ is possible only if either $a+b \equiv 0 \pmod{4}$ and $a,b \equiv 0 \pmod{2}$, or $a+b \equiv 2 \pmod{4}$ and $a,b \equiv 1 \pmod{2}$. Then the result in part (iii) of Theorem \ref{alp_rels3}
in this case follows from Theorem \ref{alp_rels_2} (ii).
\end{proof}

\section{Proof of Theorem \ref{mainth}} \label{proof1_1}

In this section we show that when $\alpha_{8n+a+3b}$ is nonzero then it is asymptotically larger than $O(n^{(a+b-1)/4})$, more precisely, in Lemma \ref{lem4_1} we prove that $n^{(a+b-1)/4}=o(\alpha_{8n+a+3b})$, as $n\rightarrow\infty$, whenever $a>1$ and $a+b\geq 4$. We then use this fact to finish the proof of Theorem \ref{mainth}. We start by establishing a uniform lower bound for our generalized divisor functions.

\begin{lemma} \label{lem_10}
Let $\epsilon, \psi \in \{\chi_1, \chi_{-3}, \chi_{-4}, \chi_{12}\}$ be two  Dirichlet characters with coprime conductors. Let $\omega(n)$ denote the number of prime divisors of 
$n= \prod_{p\mid n} p^{e_p}$. Then, for integer $k\geq 1$ and non-negative integers $n$ such that ${\rm gcd}(n, 6)=1$, we have
\bals
\lvert \sigma_{k}(\epsilon,\psi;n) \rvert \geq n^k \left( \frac{3}{5} \right)^{\omega(n)}.
\nals
\end{lemma}
\begin{proof}
We have
\bals
\lvert \sigma_k(\epsilon,\psi;n) \rvert &= \left\vert \prod_{p \mid n} \frac{ \left( \psi(p) p^k \right)^{e_p+1} - \left( \epsilon(p) \right)^{e_p+1}  }{ \psi(p) p^k  - \epsilon(p) } \right\vert 
= n^k  \left\vert \prod_{p \mid n} \frac{ 1 - \left( \frac{\epsilon(p)}{\psi(p) p^k} \right)^{e_p+1}  }{ 1  - \frac{\epsilon(p)}{\psi(p) p^k} } \right\vert \\
& \geq n^k  \prod_{p \mid n} \left(1 - \frac{1}{p^k} - \frac{1}{p^{2k}} - \cdots\right)\left(1-\frac{1}{p^k} \right)
 = n^k  \prod_{p \mid n} \frac{p^k-2}{p^k} \\
& \geq n^k  \prod_{p \mid n} \frac{3}{5} = n^k \left( \frac{3}{5} \right)^{\omega(n)}.
\nals\end{proof}

From now on we let $8n+a+3b=2^{e_2} 3^{e_3} n_{2,3}$ with $\gcd(n_{2,3},6)=1$. The following statement is a direct consequence of Lemma \ref{lem_10}.
\begin{lemma} \label{lem_11} Let $a+b \geq4$ be even and let $\epsilon$ and $\psi$ be as given in Lemma \ref{lem_10}. Then  we have
\bals
& \lim_{n \rightarrow \infty} \left\vert \frac{ 2^{e_2(a+b-2)/2} 3^{e_3(a+b-2)/2} \sigma_{\frac{a+b-2}{2}}(\epsilon,\psi; n_{2,3})}{n^{(a+b-1)/4}} \right\vert = \infty.
\nals
\end{lemma}

\begin{proof} We note that by \cite[Theorem 11]{robin} for $n \geq 3$ we have
\bals
\omega(n) \leq 1.38402 \cdot \frac{\log n}{\log \log n}.
\nals
Now the proof is straightforward using this inequality, Lemma \ref{lem_10}, and the condition $a+b\geq4$.
\end{proof}

\begin{lemma} \label{lem4_1} Let $1 < a \in \nn$, $b \in \nn_0$ with $a+b \geq 4$ even. Then we have
\bals
\lim_{n \rightarrow \infty} \frac{n^{(a+b-1)/4}}{\alpha_{8n+a+3b}}=0.
\nals
\end{lemma}

\begin{proof} Recall that $8n+a+3b=2^{e_2} 3^{e_3} n_{2,3} = \prod_{p \mid n} p^{e_p}$, where $\gcd(n_{2,3},6)=1$. We consider cases.

Case 1: Let $a,b \in \nn_0$ be such that $a+b=4k\geq 4$, and $a,b$ be both even. If $e_2=2$, then by employing Proposition \ref{thchi1},  {\eqref{multi}}, and Lemma \ref{lemma3_2} for $r=2$, we get
\bal
\frac{B_{2k}}{-4k}\alpha_{8n+a+3b} & = \frac{- \left(  \left( (-3)^{a/2} \left( 3^{2k-1} - 1 \right)+ 2 \cdot 3^{2k-1} \right)    -  {\left(  3^{2k} -1 \right)}/{(3^{(2k-1)e_3})} \right) }{2^{8k-3}(2^{2k}-1)(3^{2k}-1)(3^{2k-1}-1)}  \nonumber \\
& \quad \times 2^{(2k-1)e_2} 3^{(2k-1)e_3} \sigma_{2k-1}(\chi_1,\chi_1; n_{2,3}); \label{est1}
\nal
if $e_2=3$, then by employing Proposition \ref{thchi1}, \eqref{multi}, and Lemma \ref{lemma3_2} for $r=3$, we deduce
\bal
\frac{B_{2k}}{-4k}\alpha_{8n+a+3b} & =   \frac{ \left( (-3)^{a/2} \left( 3^{2k-1} - 1 \right)+ 2 \cdot 3^{2k-1} \right)    -  {\left(  3^{2k} -1 \right)}/ {3^{(2k-1)e_3}}  }{2^{8k-3}(2^{2k}-1)(3^{2k}-1)(3^{2k-1}-1)}    \nonumber \\
& \quad \times 2^{(2k-1)e_2} 3^{(2k-1)e_3}\sigma_{2k-1}(\chi_1,\chi_1; n_{2,3}); \label{est2}
\nal
if $e_2 \geq 4$,  then from \eqref{eq3_1} and \eqref{multi} we have again \eqref{est2}.

Case 2: Let $a,b \in \nn_0$ be such that $a+b=4k+2\geq 6$, and $a,b$ be both odd. If $e_2=2$, then by employing Proposition \ref{thchi3}, \eqref{multi}, and Lemma \ref{lemma3_2} for $r=2$, we get
\bal
\frac{B_{2k+1,\chi_{-3}}}{-(4k+2)} \alpha_{8n+a+3b} & =  \frac{ 1/3^{2ke_3} + (-3)^{(a-1)/2}  \chi_{-3}(n_{2,3}) }{ 2^{8k+1} ({2}^{2k+1}+1)}   \nonumber  \\
& \quad  \times 2^{2ke_2} 3^{2ke_3}\sigma_{2k}(\chi_{1},\chi_{-3};n_{2,3}); \label{est4}
\nal
if $e_2=3$, then by employing Proposition \ref{thchi3}, \eqref{multi}, and Lemma \ref{lemma3_2} for $r=3$, we deduce
\bal
\frac{B_{2k+1,\chi_{-3}}}{-(4k+2)} \alpha_{8n+a+3b} & = (-1)^{e_2+1} \frac{ 1/3^{2ke_3} + (-3)^{(a-1)/2}  \chi_{-3}(2^{e_2} n_{2,3}) }{  2^{8k+1} ({2}^{2k+1}+1)}  \nonumber  \\
& \quad  \times  2^{2ke_2} 3^{2ke_3} \sigma_{2k}(\chi_{1},\chi_{-3};n_{2,3}); \label{est6}
\nal
if $e_2 \geq 4$,  then, from Proposition \ref{thchi3},
\eqref{multi}, and Lemma \ref{chi-4} (ii), we have again \eqref{est6}.

Case 3: Let $a,b \in \nn_0$ be such that $a+b=4k+2\geq 6$, and $a,b$ be both even. Then, noting that $2 ~\Vert~ 8n+a+3b$, from Proposition \ref{thchi4}
and \eqref{multi} we have
\bal
\frac{B_{2k+1,\chi_{-4}}}{-(4k+2)} \alpha_{8n+a+3b} & =    \frac{  (-1)^{e_3} \left( -3^{2k} \left( (-3)^{a/2} + 2 \right) - (-3)^{a/2} \right) + (3^{2k+1} +1)/3^{2ke_3} }{  2^{6k+1} (3^{2k} + 1)(3^{2k+1}+1)}   \nonumber  \\
&  \quad \times  2^{2ke_2} 3^{2ke_3} \sigma_{2k}(\chi_{1},\chi_{-4};n_{2,3}). \label{est7}
\nal

Case 4: Let $a,b \in \nn_0$ be such that $a+b=4k\geq 4$, and $a,b$ be both odd. Then, noting that $2 ~\Vert~ 8n+a+3b$, from \eqref{eq3_9}, \eqref{multi}, and Lemma \ref{chi-4} (ii), we have
\bal
\frac{B_{2k,\chi_{-12}}}{-4k} \alpha_{8n+a+3b} & =   \frac{ 1/3^{(2k-1)e_3} -  (-1)^{e_3} (-3)^{(a-1)/2} \chi_{-3}(n_{2,3}) }{  2^{6k-2}}   \nonumber   \\
& \quad \times 2^{(2k-1)e_2} 3^{(2k-1)e_3} \sigma_{2k-1}( \chi_1,\chi_{12}; n_{2,3}). \label{est8}
\nal

Now we observe that 
\begin{itemize}
\item[] \eqref{est1}--\eqref{est2} is $0$ if only if $e_3=0$ and $a=0$; 
\item[] \eqref{est4}--\eqref{est6} is $0$ if only if $e_3=0$, $a=1$, and $2^{e_2} n_{2,3} \equiv 2 \pmod{3}$; 
\item[] \eqref{est7} is $0$ if only if $e_3=0$ and $a=0$; and 
\item[] \eqref{est8} is $0$ if only if $e_3=0$, $a=1$ and $n_{2,3} \equiv 1 \pmod{3}$. 
\end{itemize}
That is, \eqref{est1}--\eqref{est8} are non-zero whenever $a>1$. Now the result follows by applying Lemma \ref{lem_11} to \eqref{est1}--\eqref{est8}.
\end{proof}

\begin{proof}[Proof of {\rm Theorem \ref{mainth}}]
(iii) Let $\alpha_n$, $\beta_n$, $\gamma_n$, and $\gamma^\prime_n$ be as in \eqref{psiab} and \eqref{phiab}. Note that by \cite[p. 314]{cohenbook}, we have
\bal
\gamma_{n} = O_\epsilon(n^{(a+b-2)/4+\epsilon}) \mbox{ and } \gamma'_{n} = O_\epsilon(n^{(a+b-2)/4+\epsilon}), \label{eq44}
\nal
for any $\epsilon>0$.
Then when $a+3b \equiv 0 \pmod{4}$, by Lemma \ref{lem4_1}, \eqref{eq44} with $\epsilon=1/4$, and \eqref{alp_rels1}, we have
\bals
 \lim_{n \rightarrow \infty} \frac{N^*(a,b;8n+a+3b) }{\widetilde{N}(a,b;8n+a+3b)} &= \lim_{n \rightarrow \infty} \frac{2^{a+b} \left(\alpha_{8n+a+3b} + \gamma_{8n+a+3b}\right) }{\beta_{8n+a+3b} + \gamma'_{8n+a+3b}-\beta_{2n+(a+3b)/4} - \gamma'_{2n+(a+3b)/4}}\\
& \quad = \lim_{n \rightarrow \infty} \frac{\ds 2^{a+b} \left( \frac{\alpha_{8n+a+3b}}{\alpha_{8n+a+3b}} + O\left(\frac{n^{(a+b-1)/4}}{ \vert \alpha_{8n+a+3b} \vert} \right) \right) }{ \ds \frac{\beta_{8n+a+3b} -\beta_{2n+(a+3b)/4}}{\alpha_{8n+a+3b}} + O\left( \frac{n^{(a+b-1)/4}}{ \vert\alpha_{8n+a+3b} \vert}\right) }\\
& \quad =  \lim_{n \rightarrow \infty} \frac{ 2^{a+b} \alpha_{8n+a+3b} }{\beta_{8n+a+3b} -\beta_{2n+(a+3b)/4}}\\
& \quad = \frac{ 2 }{ 2^{a+b-2} + (-1)^{(a-b)/4}2^{(a+b-2)/2}}.
\nals

That is, we have \eqref{eq10_4} of Theorem \ref{mainth}. 

(i) The case when $a+3b \not\equiv 0 \pmod{8}$ can be proven similarly using Lemma \ref{lem4_1}, \eqref{eq44} with $\epsilon=1/4$, and \eqref{alp_rels2}.

(ii) The case when $a+3b \equiv 0 \pmod{8}$ can be proven similarly using Lemma \ref{lem4_1}, \eqref{eq44} with $\epsilon=1/4$, and \eqref{alp_rels4}.

\end{proof}


\begin{abib}

\bib{ACH}{article}{
   author={Adiga, Chandrashekar},
   author={Cooper, Shaun},
   author={Han, Jung Hun},
   title={A general relation between sums of squares and sums of triangular
   numbers},
   journal={Int. J. Number Theory},
   volume={1},
   date={2005},
   number={2},
   pages={175--182},
   issn={1793-0421},
   review={\MR{2173377}},
   doi={10.1142/S1793042105000078},
}

%

\bib{rmfpaper}{article}{
   author={Aygin, Zafer Selcuk},
   title={Extensions of Ramanujan-Mordell formula with coefficients 1 and
   $p$},
   journal={J. Math. Anal. Appl.},
   volume={465},
   date={2018},
   number={1},
   pages={690--702},
   issn={0022-247X},
   review={\MR{3806725}},
   doi={10.1016/j.jmaa.2018.05.033},
}

%

\bib{sqfreepaper}{article}{
   author={Aygin, Zafer Selcuk},
   title={On Eisenstein series in $M_{2k}(\Gamma_0(N))$ and their
   applications},
   journal={J. Number Theory},
   volume={195},
   date={2019},
   pages={358--375},
   issn={0022-314X},
   review={\MR{3867447}},
   doi={10.1016/j.jnt.2018.06.010},
}


\bib{projections}{article}{
   title={Projections of modular forms on Eisenstein series and its application to Siegel's formula}, 
      author={Aygin, Zafer Selcuk},
      year={2021},
      journal={preprint},
      eprint={https://arxiv.org/abs/2102.04278},
      archivePrefix={arXiv},
      primaryClass={math.NT}
   pages={1--30},
}

%
%

\bib{BCH}{article}{
   author={Barrucand, P.},
   author={Cooper, S.},
   author={Hirschhorn, M.},
   title={Relations between squares and triangles},
   journal={Discrete Math.},
   volume={248},
   date={2002},
   number={1-3},
   pages={245--247},
   issn={0012-365X},
   review={\MR{1892699}},
   doi={10.1016/S0012-365X(01)00344-2},
}

%

\bib{C4}{article}{
   author={Baruah, Nayandeep Deka},
   author={Cooper, Shaun},
   author={Hirschhorn, Michael},
   title={Sums of squares and sums of triangular numbers induced by
   partitions of 8},
   journal={Int. J. Number Theory},
   volume={4},
   date={2008},
   number={4},
   pages={525--538},
   issn={1793-0421},
   review={\MR{2441789}},
   doi={10.1142/S179304210800150X},
}

%
%

\bib{C7}{article}{
   author={Baruah, Nayandeep Deka},
   author={Kaur, Mandeep},
   author={Kim, Mingyu},
   author={Oh, Byeong-Kweon},
   title={Proofs of some conjectures of Z. -H. Sun on relations between sums
   of squares and sums of triangular numbers},
   journal={Indian J. Pure Appl. Math.},
   volume={51},
   date={2020},
   number={1},
   pages={11--38},
   issn={0019-5588},
   review={\MR{4076194}},
   doi={10.1007/s13226-020-0382-z},
}

%

\bib{C12}{article}{
   author={Bateman, Paul T.},
   author={Datskovsky, Boris A.},
   author={Knopp, Marvin I.},
   title={Sums of squares and the preservation of modularity under
   congruence restrictions},
   conference={
      title={Symbolic computation, number theory, special functions, physics
      and combinatorics},
      address={Gainesville, FL},
      date={1999},
   },
   book={
      series={Dev. Math.},
      volume={4},
      publisher={Kluwer Acad. Publ., Dordrecht},
   },
   date={2001},
   pages={59--71},
   review={\MR{1880079}},
   doi={10.1007/978-1-4613-0257-5\_4},
}

%

\bib{C10}{article}{
   author={Bateman, Paul T.},
   author={Knopp, Marvin I.},
   title={Some new old-fashioned modular identities},
   note={Paul Erd\H{o}s (1913--1996)},
   journal={Ramanujan J.},
   volume={2},
   date={1998},
   number={1-2},
   pages={247--269},
   issn={1382-4090},
   review={\MR{1642881}},
   doi={10.1023/A:1009782529605},
}

\bib{cohenbook}{book}{
   author={Cohen, Henri},
   author={Str\"{o}mberg, Fredrik},
   title={Modular forms},
   series={Graduate Studies in Mathematics},
   volume={179},
   note={A classical approach},
   publisher={American Mathematical Society, Providence, RI},
   date={2017},
   pages={xii+700},
   isbn={978-0-8218-4947-7},
   review={\MR{3675870}},
   doi={10.1090/gsm/179},
}

%

\bib{CH}{article}{
   author={Cooper, Shaun},
   author={Hirschhorn, Michael},
   title={A combinatorial proof of a result from number theory},
   journal={Integers},
   volume={4},
   date={2004},
   pages={A9, 4},
   issn={1553-1732},
   review={\MR{2079847}},
}

%

%


\bib{cooperbook}{book}{
   author={Cooper, Shaun},
   title={Ramanujan's theta functions},
   publisher={Springer, Cham},
   date={2017},
   pages={xviii+687},
   isbn={978-3-319-56171-4},
   isbn={978-3-319-56172-1},
   review={\MR{3675178}},
   doi={10.1007/978-3-319-56172-1},
}

%

\bib{cooperrmf}{article}{
   author={Cooper, Shaun},
   author={Kane, Ben},
   author={Ye, Dongxi},
   title={Analogues of the Ramanujan-Mordell theorem},
   journal={J. Math. Anal. Appl.},
   volume={446},
   date={2017},
   number={1},
   pages={568--579},
   issn={0022-247X},
   review={\MR{3554744}},
   doi={10.1016/j.jmaa.2016.09.004},
}

\bib{Kohler}{book}{
   author={K\"{o}hler, G\"{u}nter},
   title={Eta products and theta series identities},
   series={Springer Monographs in Mathematics},
   publisher={Springer, Heidelberg},
   date={2011},
   pages={xxii+621},
   isbn={978-3-642-16151-3},
   review={\MR{2766155}},
   doi={10.1007/978-3-642-16152-0},
}

%

%



\bib{ORW}{article}{
   author={Ono, Ken},
   author={Robins, Sinai},
   author={Wahl, Patrick T.},
   title={On the representation of integers as sums of triangular numbers},
   journal={Aequationes Math.},
   volume={50},
   date={1995},
   number={1-2},
   pages={73--94},
   issn={0001-9054},
   review={\MR{1336863}},
   doi={10.1007/BF01831114},
}

%
%

%

%

%

%
%

%

\bib{robin}{article}{
   author={Robin, Guy},
   title={Estimation de la fonction de Tchebychef $\theta $ sur le $k$-i\`eme
   nombre premier et grandes valeurs de la fonction $\omega (n)$ nombre de
   diviseurs premiers de $n$},
   language={French},
   journal={Acta Arith.},
   volume={42},
   date={1983},
   number={4},
   pages={367--389},
   issn={0065-1036},
   review={\MR{736719}},
   doi={10.4064/aa-42-4-367-389},
}

%

%

%

%

\bib{stein}{book}{
   author={Stein, William},
   title={Modular forms, a computational approach},
   series={Graduate Studies in Mathematics},
   volume={79},
   note={With an appendix by Paul E. Gunnells},
   publisher={American Mathematical Society, Providence, RI},
   date={2007},
   pages={xvi+268},
   isbn={978-0-8218-3960-7},
   isbn={0-8218-3960-8},
   review={\MR{2289048}},
   doi={10.1090/gsm/079},
}

%
%

\bib{SunPaper}{article}{
   author={Sun, Zhi-Hong},
   title={Some relations between $t(a,b,c,d;n)$ and $N(a,b,c,d;n)$},
   journal={Acta Arith.},
   volume={175},
   date={2016},
   number={3},
   pages={269--289},
   issn={0065-1036},
   review={\MR{3557125}},
   doi={10.4064/aa8418-5-2016},
}

%
%
%
%

\end{abib}

\end{document}